\documentclass{amsart}

\usepackage[utf8]{inputenc}
\usepackage[T1]{fontenc}

\usepackage[backend=biber, giveninits=true, style=alphabetic]{biblatex}
\AtEveryBibitem{\clearfield{issn} \clearfield{url} \clearfield{isbn} \clearfield{doi} \clearlist{language}}
\DeclareFieldFormat[article, inbook, incollection, inproceedings, thesis]{title}{{#1\isdot}}
\DeclareFieldFormat[article, book]{volume}{\mkbibbold{#1}}
\DeclareFieldFormat[article]{number}{\mkbibparens{#1}}
\DeclareFieldFormat{pages}{#1}
\DeclareLabelalphaNameTemplate{
  \namepart[use=true, pre=true, strwidth=1, compound=true]{prefix}
  \namepart[compound=true]{family}
}
\renewbibmacro*{volume+number+eid}{%
	\printfield{volume}%
	\setunit*{\addnbthinspace}
	\printfield{number}%
	\setunit{\addcomma\space}%
	\printfield{eid}
}
\renewbibmacro{in:}{}
\usepackage{xurl} 
\usepackage{amsmath,amsthm,amssymb,amsfonts}
\usepackage{mathtools} 
\usepackage{tikz-cd}
\usepackage{bbold} 
\usepackage{MnSymbol} 
\usepackage{stmaryrd} 
\usepackage[inline]{enumitem}
\usepackage{caption}
\usepackage{anyfontsize} 

\usepackage{xcolor}
\definecolor{IndigoBlue}{RGB}{4,35,96}

\usepackage[colorlinks]{hyperref}
\hypersetup{
    colorlinks=true,
	linkcolor=IndigoBlue,
	citecolor=IndigoBlue,
	urlcolor=IndigoBlue,
	pdftitle = {Window equivalences via categorical sl(2) actions},
	pdfauthor = {Wei Tseu}
}

\theoremstyle{plain}
\newtheorem{thm}{Theorem}[section]
\newtheorem{lem}[thm]{Lemma}
\newtheorem{prop}[thm]{Proposition}
\newtheorem{cor}[thm]{Corollary}

\newtheorem{thm*}{Theorem}

\newtheoremstyle{note}%
  {}{}
  {}{}
  {\itshape}{.}
  { }{}
\theoremstyle{note}
\newtheorem{defn}[thm]{Definition}
\newtheorem{rmk}[thm]{Remark}
\newtheorem{eg}[thm]{Example}

\newcommand\simrightarrow{\stackrel{\textstyle\sim}{\smash{\longrightarrow}\rule{0pt}{0.3ex}}} 

\newcommand{\on}[1]{\operatorname{#1}}
\newcommand{\st}{\,\big\vert\,}
\newcommand{\Gr}{\mathrm{Gr}}
\newcommand{\TGr}[2]{T^{*}\Gr(#1,#2)}
\newcommand{\Fl}{\mathrm{Fl}}

\newcommand{\Hom}{\on{Hom}}
\newcommand{\Ext}{\on{Ext}}
\newcommand{\Sym}{\on{Sym}}
\newcommand{\End}{\on{End}}

\makeatletter 
\newcommand{\extp}{\@ifnextchar^\@extp{\@extp^{\,}}}
\def\@extp^#1{\mathop{\bigwedge\nolimits^{\!#1}}}
\makeatother

\newcommand{\bC}{\mathbb{C}}
\newcommand{\bR}{\mathbb{R}} 
\newcommand{\bP}{\mathbb{P}}

\newcommand{\be}{\mathbb{e}} 
\newcommand{\bbf}{\mathbb{f}} 
\newcommand{\bS}{\mathbb{S}} 
\newcommand{\bT}{\mathbb{T}} 
\newcommand{\bW}{\mathbb{W}} 

\newcommand{\GL}{\mathrm{GL}}
\newcommand{\gl}{\mathfrak{gl}} 
\newcommand{\fsl}{\mathfrak{sl}}
\newcommand{\fg}{\mathfrak{g}} 
\newcommand{\fZ}{\mathfrak{Z}} 

\newcommand{\dI}{\dot{I}}
\newcommand{\dJ}{\dot{J}}
\newcommand{\dR}{\dot{R}}

\newcommand{\ddI}{\ddot{I}}
\newcommand{\ddJ}{\ddot{J}}
\newcommand{\ddR}{\ddot{R}}

\newcommand{\cO}{\mathcal{O}} 
\newcommand{\cW}{\mathcal{W}} 
\newcommand{\cT}{\mathcal{T}}
\newcommand{\cF}{\mathcal{F}} 
\newcommand{\cV}{\mathcal{V}}
\newcommand{\cK}{\mathcal{K}}

\newcommand{\cC}{\mathcal{C}}
\newcommand{\cA}{\mathcal{A}}
\newcommand{\cG}{\mathcal{G}}

\newcommand{\la}{\lambda}
\newcommand{\te}{\theta}

\newcommand{\Db}{\mathnormal{D}^{\mathrm{b}}} 
\newcommand{\MF}{\mathrm{MF}} 

\title{Window equivalences via categorical $\fsl(2)$ actions} 
\author{Wei Tseu}

\date{19 June 2025}

\begin{document}

\begin{abstract}
We identify two distinct approaches to the derived equivalence for the stratified Mukai flop of cotangent bundles of Grassmannians---one induced by the geometric categorical $\fsl(2)$ action, and the other through the magic window category of graded matrix factorizations on the gauged Landau--Ginzburg model---via the Kn{\"o}rrer periodicity. 
\end{abstract}

\maketitle

\section{Introduction}
Let $\Gr(k,\bC^{N})$ and $\Gr(\bC^{N},k)$ be the Grassmannians of $k$-dimensional subspaces and $k$-dimensional quotients of $\bC^{N}$ respectively. 
As two symplectic resolutions of a nilpotent orbit closure $\mathfrak{N} \subset \End(\bC^{N})$ (see (\ref{eq:nil})), the cotangent bundles $\TGr{k}{\bC^{N}}$ and $\TGr{\bC^{N}}{k}$ are related by the \textit{stratified Mukai flop}. 
That is to say, by blowing up the zero sections of the two cotangent bundles, one obtains a common roof 
\begin{equation*}
    \TGr{k}{\bC^{N}} \longleftarrow  \fZ^{(0)} \longrightarrow  \TGr{\bC^{N}}{k},
\end{equation*}
where the exceptional divisor is identified with the incidence variety in the product $\Gr(k,\bC^{N})\times \Gr(\bC^{N},k)$. 
This irreducible variety $\fZ^{(0)}$ is in fact a natural correspondence; it contains the largest common open subset of $\TGr{k}{\bC^{N}}$ and $\TGr{\bC^{N}}{k}$ over the singularity $\mathfrak{N}$. 
However, it turns out \cite{Nami2}
that $\fZ^{(0)}$ is not the right place to induce a natural equivalence of the bounded derived categories of coherent sheaves 
\begin{equation}\label{eq:flopeq}
    \Db(\TGr{k}{\bC^{N}}) \simrightarrow \Db(\TGr{\bC^{N}}{k}).
\end{equation}
The correspondence $\fZ^{(0)}$ yet also appears in the fibre product of the blow-downs of the zero sections 
\begin{equation*}
    \TGr{k}{\bC^{N}} \longrightarrow \mathfrak{N}  \longleftarrow \TGr{\bC^{N}}{k},
\end{equation*}
which consists of $k+1$ equidimensional components $\fZ^{(0)},\cdots,\fZ^{(i)}, \cdots, \fZ^{(k)}$.
Here each component $\fZ^{(i)}$ is characterized by the rank of its cotangent vectors in $\mathfrak{N}$ as well as the codimension of the incidence locus in its zero section. 
For example, the last one $\fZ^{(k)} \cong \Gr(k,\bC^{N}) \times \Gr(\bC^{N},k)$ is the `deepest' component of zero cotangent vectors.  
We refer to \cite[\S 2.2]{Cautis} for a detailed description of these correspondences. 

In a series of papers \cite{CKL-Duke,CKL-sl2,Cautis},  Cautis, Kamnitzer and Licata developed the theory of geometric categorical $\fsl(2)$ actions and applied it to the above stratified Mukai flop to obtain a natural equivalence (\ref{eq:flopeq}).
We will briefly recall the construction of this categorical action in \S\ref{sec:sl2}. 
Roughly speaking, the equivalence is given by the convolution $\bT = \on{Conv}(\Theta)$ of a `Rickard complex', which is of the form
\[
\Theta = \left\{ \Theta^{(k)}[-k]\longrightarrow \cdots \longrightarrow \Theta^{(i)}[-i] \longrightarrow \cdots \longrightarrow \Theta^{(0)} \right\},
\]
where each $\Theta^{(i)}$ is an integral functor induced by a sheaf supported on $\fZ^{(i)}$, with $\Theta^{(0)}$ induced by $\cO_{\fZ^{(0)}}$ in particular.
The differentials in this complex are essentially defined by counits of the adjunctions arising from the  categorical $\fsl(2)$ action.
The entire complex categorifies the (Lefschetz) decomposition of the reflection element of $\mathrm{SL}(2)$, acting on the weight spaces of an irreducible representation of $\fsl(2)$ (see \cite[\S 1]{CKL-sl2}). 

This paper aims to interpret the above Rickard complex from a different geometric perspective. 
We will demonstrate how the convolution perfectly corrects $\cO_{\fZ^{(0)}}$ by mutating it into the right kernel to induce an equivalence (\ref{eq:flopeq}). 
The meaning of the term `correction' will become clear as we proceed.

Another way to understand the stratified Mukai flop is to view the cotangent bundles of Grassmannians as symplectic quotients. 
In \S\ref{sec:Nak}, we recall the definition of Nakajima quiver varieties and realize $\TGr{k}{\bC^{N}}$ and $\TGr{\bC^{N}}{k}$ as hyperk\"{a}hler quotients $\mu^{-1}(0)\sslash_{\te_{\pm}} G$.
Here $\mu:X\to \fg^{\vee}$ is the moment map associated to a symplectic representation
\[
X = \Hom(V,\bC^{N}) \oplus \Hom(\bC^{N},V)
\]
of $G = \GL(V)$, where $V$ is a $k$-dimensional vector space and $\te_{\pm}$ are the two respective stability conditions for taking the GIT quotients. 
As the moment map defines a section to the dual Lie algebra bundle $\fg^{\vee}$ over $X$, there is a superpotential function 
\[
w: X \oplus \fg \longrightarrow \bC
\]
defined on the Landau--Ginzburg model $X \oplus \fg$ via the natural pairing (\S\ref{sec:mf}). 
Then, the Kn{\"o}rrer periodicity (Theorem~\ref{thm:KP}) establishes an equivalence of triangulated categories 
\[
\Psi: \Db(\mu^{-1}(0)\sslash_{\te_{\pm}} G) \simrightarrow \MF_{G\times \bC^{\times}}(X^{\te_{\pm}\text{-ss}} \times \fg,w)
\]
to the derived category of graded matrix factorizations\footnote{For readers not familiar with matrix factorizations, it may be helpful to think of them as equivariant coherent sheaves over the singular locus $w^{-1}(0)$ modulo all perfect complexes.} on the gauged LG models.  

The magic window theory (\S\ref{sec:magic}) of Halpern-Leistner and Sam \cite{HLS} tells us that  for every choice of the window parameter $\delta$, there is a magic window subcategory (Definition~\ref{def:winmf})
\[
\cW_{\delta} \subset \MF_{G\times \bC^{\times}}(X\oplus \fg,w),
\]
such that the inclusions $\iota_{\pm}: X^{\te_{\pm}\text{-ss}} \times \fg \hookrightarrow X\oplus \fg$ induce equivalences
\[
\bW_{\delta}: \MF_{G\times \bC^{\times}}(X^{\te_{+}\text{-ss}} \times \fg,w) \xrightarrow{(\iota_{+}^{*})^{-1}} \cW_{\delta} \xrightarrow{\,\,\,\,\,\,\iota_{-}^{*} \,\,\,\,\,\,} \MF_{G\times \bC^{\times}}(X^{\te_{-}\text{-ss}}  \times \fg,w). 
\]
In the current example, there are $\mathbb{Z}$-choices of $\delta$ up to equivalences, and the corresponding magic windows are also indexed by half-open intervals $[n,n+N)$, $n \in \mathbb{Z}$; see Example~\ref{eg:window}.   
This means the magic window subcategory $\cW_{\delta} = \cW_{[n,n+N)}$ is generated by matrix factorizations whose components are
equivariant vector bundles of the form $V_{\la}\otimes \cO_{X\oplus \fg}$, where $V_{\la}$ is the irreducible representation of $G\cong \GL(k)$ of highest weight $\la = (\la_{i})_{i=1}^{k}$  satisfying $\la_{i}\in [n,n+N)$.

The main result (Theorem~\ref{thm:main}) of this paper identifies the twisted equivalence $\Psi \bT \Psi^{-1}$ with one of the window equivalences $\bW_{\delta}$. 

\begin{thm*}
    There is a window equivalence $\bW_{\delta}$ such that the following diagram of triangulated category equivalences commutes
    \begin{equation*}
        \begin{tikzcd}
            \MF_{G \times \bC^{\times}}(X^{\te_{+}\text{-ss}} \times \fg, w) \arrow[r, "\bW_{\delta}"] & 
            \MF_{G \times \bC^{\times}}(X^{\te_{-}\text{-ss}} \times \fg, w)  \\
            \Db(\TGr{k}{\bC^{N}}) \arrow[r, "\bT"] \arrow[u, "\Psi"] & 
            \Db(\TGr{\bC^{N}}{k}) \arrow[u, "\Psi"].
        \end{tikzcd}
    \end{equation*}
\end{thm*}

Note that the ambient space $X\oplus \fg$ is much simpler than the hyperk\"{a}hler quotients, so it is reasonable to translate $\bT$  by the Kn{\"o}rrer periodicity $\Psi$ and compare it to the window equivalences over the LG models. 

In \S\ref{sec:Hecke}, we recall the definition of Hecke correspondences,  which are the spaces supporting the kernels of the defining integral functors $\be, \bbf$ of the categorical $\fsl(2)$ action. 
We also write down the LG models and the Kn{\"o}rrer periodicity for these Hecke correspondences.
After stating two corollaries of the main theorem in \S\ref{sec:main}, we find the matrix factorization kernels that induce the corresponding twisted functors $\Psi \be \Psi^{-1}$, $\Psi \bbf \Psi^{-1}$ and $\Psi \Theta \Psi^{-1}$ in the next subsection \S\ref{sec:kernel}. 

The remaining three subsections are devoted to the proof of the main theorem. 
In \S\ref{sec:1}, we first extend the kernels of $\Psi \Theta \Psi^{-1}$ naturally
from $(X^{\te_{+}\text{-ss}}\times \fg) \times (X^{\te_{-}\text{-ss}}\times \fg)$ to 
$(X^{\te_{+}\text{-ss}}\times \fg) \times (X \oplus \fg)$
along the inclusion $\on{id}\times \iota_{-}$.
Next, we verify that the two conditions that uniquely characterize a kernel inducing the window equivalence $\bW_{\delta}$ are satisfied by the convolution of this complex of extended kernels. 
Relying on Corollary~\ref{cor:resoln} from the end, this concludes the proof of the main theorem.  
The corollary then follows from the grade restriction rule for the convolution kernel, which is discussed in the last two subsections. 
We recall the setup of grade restriction windows in \S\ref{sec:2} and verify the grade restriction rule for our kernels in \S\ref{sec:3}.
The highlight is Proposition~\ref{prop:main}, where we demonstrate how the convolution of the Rickard complex iteratively eliminates the parts that are outside of the grade restriction window via a Gram--Schmidt type algorithm, and ultimately results in an object inside the window. 

The takeaway is that the kernel $\cO_{\fZ^{(0)}}$, translated by the Kn{\"o}rrer periodicity into the matrix factorization context, satisfies the expected window condition only after being corrected by the Rickard complex.

\subsection*{Acknowledgments}
I am deeply indebted to Travis Schedler and Ed Segal; without their help and support, I would not have survived this project. 
Thanks also go to Richard Thomas for his comments on an earlier draft and encouragement, and Daniel Halpern-Leistner and Yukinobu Toda for kindly answering my questions via email.
This work was supported by the Engineering and Physical Sciences Research Council [EP/S021590/1]; The EPSRC Centre for Doctoral Training in Geometry and Number Theory (The London School of Geometry and Number Theory), University College London. 

\section{Preliminaries}

All functors are derived and all sheaves are considered equivariant in this paper. 

\subsection{Nakajima quiver varieties}\label{sec:Nak}
The quiver of interest in this paper consists of a single vertex and no edges. 
For $k,N\in\mathbb{N}$, consider the following space of double framed quiver representations
\begin{equation*}
    X_{k} = \Hom(V_{k},\bC^{N})\oplus \Hom(\bC^{N}, V_{k}),
\end{equation*}
where $V_{k}$ is a vector space of dimension $k$ and $\bC^{N}$ is the fixed framing.
The space $X_{k}$ is a symplectic representation of the gauge group 
\begin{equation*}
    G_{k} = \GL(V_{k})
\end{equation*}
for the action $g\cdot (a,b) = (ag^{-1}, gb)$ where $g\in G_{k}$ and $(a,b)\in X_{k}$.
For a stability parameter $\theta \in \mathbb{Z}$, the \textit{Nakajima quiver variety} \cite{Nak94} is defined as the GIT quotient 
\begin{equation*}
    \mathfrak{M}_{\theta}(k,N) = \mu^{-1}(0) \sslash_{\chi_{\theta}} G_{k},
\end{equation*}
where $\mu: X_{k} \to \mathfrak{gl}_{k}^{\vee}$ is the moment map and $\chi_{\theta}$ is the character $\det(g)^{-\theta}$ of $G_{k}$. 

Assume $k\leq N$ and take $\te_{\pm} = \pm 1$, then the semi-stable loci are described as
\begin{equation*}
    X^{\pm}_{k} = X_{k}^{\te_{\pm}\text{-ss}} = 
    \left\{ 
    (a,b)\in X_{k} \st \on{rk}(a)=k \,\,  (\text{resp.\ }\on{rk}(b)=k) 
    \right\}. 
\end{equation*}
The moment map condition $\mu(a,b) = ba =0$ further implies $a(V_{k}) \subseteq \ker(b)$. 
So, the quiver varieties $\mathfrak{M}_{\theta_{\pm}}(k,N)$ are isomorphic to the cotangent bundles $\TGr{k}{\bC^{N}}$ and $\TGr{\bC^{N}}{k}$ respectively. 
In fact, they are resolutions of the categorical quotient $\mathfrak{M}_{0}(k,N)$, which is identified with the nilpotent orbit closure
\begin{equation} \label{eq:nil}
    \mathfrak{N} = \left\{
    f=ab \in \End(\bC^{N}) \st f^{2} =0,\quad  \on{rk}(f)\leq k     \right\}.
\end{equation}

Quiver varieties with reflecting stability parameters (and dimension vectors) are also related by the Lusztig--Maffei--Nakajima~(LMN) isomorphism. 
In the current example, it is simply the canonical isomorphism
\begin{equation}\label{eq:LMN}
    \TGr{N-k}{\bC^{N}} = \mathfrak{M}_{\theta_{+}}(N-k,N) \simrightarrow \mathfrak{M}_{\theta_{-}}(k,N) =  \TGr{\bC^{N}}{k}.
\end{equation}
In this isomorphism, the embedding $a:V_{N-k} \hookrightarrow \bC^{N}$ of a point $(a,b)$ from the left-hand side gives rise to a quotient $b^{\prime}: \bC^{N}\twoheadrightarrow V_{k}$ as its cokernel.
Then, because of the moment map condition, the map $b: \bC^{N}\to V_{N-k}$ naturally factors through this quotient $V_{k}$. 
Now, if we compose the universal map $V_{k}\to V_{N-k}$ of this factorization with $a$, we obtain a map $a^{\prime}: V_{k}\to \bC^{N}$. 
The pair $(a^{\prime},b^{\prime})$ represents the image of $(a,b)$ under the isomorphism (\ref{eq:LMN}). 

\subsection{Magic window categories}\label{sec:magic}
The derived category of a GIT quotient has a nice combinatorial description via the magic window theory \cite{HLS}.  
We refer to \cite[pp.\ 238--239]{BFK} for a historical account of the idea of `windows'.
Let $G$ be a complex reductive group with Lie algebra $\fg$, and $T$ a maximal torus of $G$. 
For a dominant weight $\la\in \mathbb{X}^{*}(T)^{+}$, there is an irreducible representation $V_{\la}$ of $G$ of highest weight $\la$.

Given a symplectic representation $X$ of $G$,  denote its set of $T$-weights by $\{v_{i}\}_{i\in I}$ (counted with multiplicity) and define the zonotope
\begin{equation*}
    {\overline{\Sigma}}_{X} = \left\{
    \sum_{i\in I} c_{i} v_{i} \st c_{i}\in [-1,0] 
    \right\} \subset \mathbb{X}^{*}(T)_{\bR}. 
\end{equation*}
Following \cite[\S 2]{HLS}, we assume the Weyl-invariants $\mathbb{X}^{*}(T)^{W}_{\bR}$ are not contained in any linear hyperplane that is parallel to a codimension one face of $ {\overline{\Sigma}}_{X\oplus \fg}$ (here $\fg$ is the adjoint representation of $G$), and ${\overline{\Sigma}}_{X\oplus \fg}$ linearly spans $\mathbb{X}^{*}(T)_{\bR}$. 
An element $\delta \in \mathbb{X}^{*}(T)^{W}_{\bR}$ is called \textit{generic} if the boundary of the shifted window 
\begin{equation*}
    \delta + \frac{1}{2}  {\overline{\Sigma}}_{X}
\end{equation*}
does not intersect the weight lattice $\mathbb{X}^{*}(T)$. 

\begin{rmk}
    The $G$-space $X\oplus \fg$ is quasi-symmetric in the sense of \cite[\S 2]{HLS}, and the magic window associated to it is actually a polytope $\overline{\nabla}$ (\cite[\S 2.2]{HLS}) defined by the weights of $X\oplus \fg$ (rather than of $X$). 
    However, in this case, the polytope $\overline{\nabla}_{X\oplus \fg}$ is equal to the zonotope $\overline{\Sigma}_{X}$ associated to $X$; see \cite[Remark 6.13]{HLS}. 
    As we will only consider quotients of spaces of the form $X\oplus \fg$, we also refer to $\overline{\Sigma}_{X}$ the magic window. 
\end{rmk}

\begin{defn}[\cite{HLS}]
    For a generic window $\delta + (1/2)  {\overline{\Sigma}}_{X}$, define the \textit{window subcategory} $\mathnormal{W}_{\delta}$ as the full subcategory of $\Db_{G}(X\oplus \fg )$ split-generated by $G$-equivariant vector bundles $V_{\la} \otimes \cO_{X\oplus \fg}$ for $\la\in \delta + (1/2)  {\overline{\Sigma}}_{X}$.
\end{defn}

\begin{eg}\label{eg:window}
    Consider the symplectic representation from \S\ref{sec:Nak}
    \begin{equation*}
        X_{k} = \Hom(V_{k}, \bC^{N}) \oplus \Hom(\bC^{N}, V_{k}).
    \end{equation*}
    It has weights $\{ \pm N\cdot e_{i}\}_{i=1}^{k}$ in terms of the standard basis $\{e_{i}\}_{i=1}^{k}$ of $V_{k}$. 
    In the space $\mathbb{X}^{*}(T)_{\bR} \cong \bR^{k}$, the original window is identified with the region 
    \begin{equation*}
        \frac{1}{2} {\overline{\Sigma}}_{X_{k}} = \left\{
        (x_{1},\cdots, x_{k})\in \bR^{k} \st -\frac{N}{2} \leq x_{i} \leq \frac{N}{2}, \forall i 
        \right\}. 
    \end{equation*}
    Shift it along the line $\mathbb{X}^{*}(T)_{\bR}^{W} = \{x_{1}=\cdots = x_{k}\} \cong \bR$ by a generic window parameter $\delta\in \bR$, then any dominant weight in this shifted window $\delta + (1/2){\overline{\Sigma}}_{X_{k}} $ is represented by a non-increasing sequence $\la_{1}\geq \cdots \geq\la_{k}$ of integers where each
    \begin{equation*}
        \la_{i} \in \left[ \lceil \delta - (N/2) \rceil, \lceil \delta + (N/2) \rceil \right). 
    \end{equation*}
    In this case, we also call this width $N$ half-open interval a \textit{window}. 
\end{eg}

A GIT quotient is called \textit{generic} if the semi-stable and stable loci coincide.

\begin{thm}[\cite{HLS}]
    For a generic window parameter $\delta$ and a generic stability condition $\theta$, the inclusion $\iota: (X\oplus\fg)^{\theta\text{-ss}}\hookrightarrow X\oplus \fg$ of the semi-stable locus induces an equivalence
    \begin{equation*}
        \iota^{*}: \mathnormal{W}_{\delta} \simrightarrow \Db((X\oplus\fg)^{\theta\text{-ss}}/G). 
    \end{equation*}
\end{thm}

\begin{cor}[\cite{HLS}]
    For a window subcategory $\mathnormal{W}_{\delta}$ and two generic stability conditions $\te_{1}, \te_{2}$, there is a categorical GIT wall-crossing equivalence 
    \begin{equation*}
        \Db((X\oplus\fg)^{\theta_{1}\text{-ss}}/G) 
        \xlongrightarrow{(\iota_{1}^{*})^{-1}} 
        \mathnormal{W}_{\delta}
        \xlongrightarrow{\,\,\,\,\,\, \iota_{2}^{*}\,\,\,\,\,\,} 
        \Db((X\oplus\fg)^{\theta_{2}\text{-ss}}/G). 
    \end{equation*}
\end{cor}

\subsection{Matrix factorizations}\label{sec:mf}
We first recall a result that relates two kinds of derived categories of coherent sheaves, respectively over the singular locus $\mu^{-1}(0)$ and the smooth space of representations of a quiver with potential.  
This is an equivalence of triangulated categories studied mathematically by \cite{Orlov04,Orlov06, Ed, Ship, Isik}, and it is known as the (derived) Kn{\"o}rrer periodicity \cite{Hira}. 

Suppose $X$ is a smooth algebraic variety equipped with a reductive group $G$-action. 
Let $E$ be a $G$-equivariant vector bundle on $X$ and $s$ a $G$-equivariant regular section of $E$. 
Define the \textit{superpotential}  function $w$ on the total space of the dual bundle $E^{\vee}$ by the natural pairing
\begin{equation*}
    w: \on{Tot}(E^{\vee}) \longrightarrow \bC, \quad (x, v) \longmapsto \langle s(x), v \rangle, \quad x\in X, v\in E^{\vee}_{x}. 
\end{equation*}
The function $w$ is $G$-invariant, yet of weight two with respect to the squared dilation of $\bC^{\times}$ on the fibres of $\on{Tot}(E^{\vee})$. 
The pair $(\on{Tot}(E^{\vee}),w)$ is known as a \textit{gauged Landau--Ginzburg model}~(LG model). 

A \textit{graded matrix factorization} on $(\on{Tot}(E^{\vee}) ,w)$ is a pair of $(G\times \bC^{\times})$-equivariant coherent sheaves $F_{0},F_{1}$ (called \textit{components}) over $\on{Tot}(E^{\vee})$, together with two $G$-equivariant and $\bC^{\times}$-weight one maps 
\begin{equation*}
    \begin{tikzcd}        
    F_{0} \arrow[r, shift left, "d_{0}"] & F_{1} \arrow[l, shift left, "d_{1}"],
    \end{tikzcd}
\end{equation*}
such that $d_{1}d_{0} = w \cdot \on{id} = d_{0} d_{1}$. 
By adopting the (cohomological) `internal' degree shift $\langle \cdot \rangle$ on the $\mathbb{Z}$-grading, the maps are denoted equivariantly
\[
F_{0} \xlongrightarrow{d_{0}} F_{1} \langle 1 \rangle \xlongrightarrow{d_{1}} F_{0} \langle 2 \rangle. 
\]
This matrix factorization is also equivalent to a  $(G\times \bC^{\times})$-equivariant coherent sheaf $F=F_{0}\oplus F_{1}$, together with a $G$-equivariant endomorphism 
\begin{equation*}
    d_{F}= \begin{pmatrix}
        0 & d_{1} \\
        d_{0} & 0
    \end{pmatrix}: F \longrightarrow F\langle 1\rangle 
\end{equation*}
such that $d_{F}^{2}=w$; see \cite[\S 2]{Ed}. 
The pair $(F,d_{F})$ is then called a curved dg sheaf.   

We may also think of matrix factorizations as twisted $\mathbb{Z}/2$-graded complexes and consider homotopy classes of cochain maps between them. 
That is to say, for two matrix factorizations $F, F^{\prime}$, we define a differential graded structure on the following complex $\Hom^{\bullet}(F,F^{\prime})$ of $G$-equivariant maps of various $\bC^{\times}$-weights
\begin{align*}
    \Hom^{2n}(F,F^{\prime})& = \Hom(F_{0},F^{\prime}_{0}\langle 2n \rangle)\oplus \Hom(F_{1},F^{\prime}_{1}\langle 2n \rangle), \\
    \Hom^{2n+1}(F,F^{\prime})& = \Hom(F_{0},F^{\prime}_{1}\langle 2n+1 \rangle)\oplus \Hom(F_{1},F^{\prime}_{0}\langle 2n+1 \rangle),
\end{align*}
by the differential $d(f) = d_{F^{\prime}}\circ f - (-1)^{\on{deg}(f)}f \circ d_{F}$. 
By taking the 0-th cohomology group $H^{0}(\Hom^{\bullet}(F,F^{\prime}))$ as the set of morphisms, one obtains the homotopy category $K_{G\times \bC^{\times}}(\on{Tot}(E^{\vee}), w)$ of matrix factorizations. 
This is a triangulated category with shift 
\[
F[1] =  \left\{ F_{1} \langle 1 \rangle \xlongrightarrow{-d_{1}} F_{0} \langle 2 \rangle \xlongrightarrow{-d_{0}} F_{1} \langle 3 \rangle \right\},
\]
and a natural construction of cones
\begin{equation*}
    \on{Cone}\left( \phi: F \to F^{\prime} \right) = \left\{ F[1] \oplus F^{\prime}, 
    \begin{pmatrix}
        -d_{1} & 0 \\ \phi_{1} & d_{0}^{\prime}
    \end{pmatrix}, 
    \begin{pmatrix}
        -d_{0} & 0 \\ \phi_{0} & d_{1}^{\prime}
    \end{pmatrix}\right\}. 
\end{equation*}
Since a matrix factorization is a complex twisted by a usually nonzero function $w$, we no longer have the notion of its cohomology. 
Instead, a matrix factorization is called \textit{acyclic} if it is in the smallest thick subcategory of $K_{G\times \bC^{\times}}(\on{Tot}(E^{\vee}), w)$ that contains all totalizations 
of exact sequences of matrix factorizations.  
By taking the Verdier quotient of the homotopy category by the thick subcategory of acyclic matrix factorizations, one reaches the \textit{derived category of graded matrix factorizations}
\begin{equation*}
    \MF_{G\times \bC^{\times}}(\on{Tot}(E^{\vee}), w). 
\end{equation*} 
Two matrix factorizations that become isomorphic in the derived category are called \textit{quasi-isomorphic}.
Like in the derived category of coherent sheaves, functors such as the  tensor product,  pullback, and  pushforward are also defined for matrix factorizations. 
We refer to \cite{BFK-2} or \cite{Hira-2} for a formal introduction.

As $s$ is regular, we can resolve the skyscraper sheaf $\cO_{s^{-1}(0)}$ by the Koszul complex $(\extp^{\bullet}E^{\vee}, s \righthalfcup)$ over $X$. 
Pull back this complex to $\on{Tot}(E^{\vee})$ along the projection, we obtain a distinguished matrix factorization
\begin{equation}\label{eq:Kmf}
    \begin{tikzcd}        
     \extp^{\bullet} E^{\vee}: \extp^{\text{top}} E^{\vee} \arrow[r, shift left, "s \righthalfcup"] & \cdots \arrow[l, shift left, "\wedge t"] \arrow[r, shift left, "s \righthalfcup"] & E^{\vee} \arrow[l, shift left, "\wedge t"] \arrow[r, shift left, "s\righthalfcup"] & \cO_{\on{Tot}(E^{\vee})}  \arrow[l, shift left, "\wedge t"],
    \end{tikzcd} 
\end{equation}
where $t$ is the tautological section of the bundle $E^{\vee}$ on $\on{Tot}(E^{\vee})$. 
Here $s$ is equivariant while $t$ is of weight two, so $E^{\vee}$ should be shifted by $\langle -1\rangle$ to make sure the differentials in both directions are of weight one. 
Forgetting the arrows of $s \righthalfcup$ in (\ref{eq:Kmf}) gives a resolution of the zero section $\cO_{X}$, up to tensoring with $\det(E^{\vee})[-\on{rk}(E)]$. 
In other words, the sheaves $\cO_{s^{-1}(0)}$ and $\cO_{X}\otimes \det(E^{\vee})[-\on{rk}(E)]$ are quasi-isomorphic as matrix factorizations; they have the common Koszul resolution (\ref{eq:Kmf}). 

\begin{thm}[Kn{\"o}rrer periodicity]\label{thm:KP}
    The Koszul factorization (\ref{eq:Kmf}) induces an equivalence of triangulated categories  
    \begin{equation*}
        \extp^{\bullet} E^{\vee} \otimes_{\cO_{s^{-1}(0)}} (-): \Db_{G}(s^{-1}(0)) \simrightarrow \MF_{G\times\bC^{\times}}(\on{Tot}(E^{\vee}), w). 
    \end{equation*}
\end{thm}

\begin{rmk}\label{rmk:kno}
    This tensor product is taken over dg schemes (the derived fibre $s^{-1}(0)$), transforming graded dg modules over the dg algebra 
    \[
    \Sym\left( \underline{0} \longrightarrow E^{\vee} 
    \xrightarrow{-s^{\vee}} \cO_{X} \right)
    \]
    into that of $\Sym( \cO_{X} \xrightarrow{s} \underline{E} 
    )$ via Mirkovi\'{c}--Riche's linear Koszul duality. 
    See \cite{Isik} for the full statement in terms of the singularity categories and \cite[\S 2.3]{Toda} for the curved dg version. 
    In practice, the Kn{\"o}rrer periodicity is given by the  composition $i_{*}\circ p^{*}$ of the direct and inverse image functors along the bundle projection $p$ and the obvious embedding $i$ as shown in the diagram (see, e.g., \cite{Hira})
    \begin{equation*}
    \begin{tikzcd}
        \MF_{G\times \bC^{\times}}(\on{Tot}(E^{\vee})|_{s^{-1}(0)}, 0) \arrow[r,"i_{*}"] & \MF_{G\times\bC^{\times}}(\on{Tot}(E^{\vee}), w) \\
        \MF_{G\times \bC^{\times}}(s^{-1}(0), 0) = \Db_{G}(s^{-1}(0)) \arrow[u,"p^{*}"].
    \end{tikzcd}
    \end{equation*}
    In the bottom line, a bounded complex of coherent sheaves is viewed as a matrix factorization of zero potential, with its original cohomological degree transferred into the internal degree (i.e.\ the $\mathbb{Z}$-grading). 
\end{rmk}

The following lemma tells the functoriality of the direct and inverse image functors with respect to the Kn{\"o}rrer periodicity. 
It will only be used later.

\begin{lem}[{\cite[\S 2.4]{Toda}}]\label{lem:fun}
    Let $(X_{i}, G_{i}, E_{i}, s_{i},w_{i}), i=1,2$ be two tuples of the above geometric data of an LG model. 
    Suppose there is a commutative diagram of equivariant morphisms
    \begin{equation*}
        \begin{tikzcd}
            E_{1} \arrow[r, "F"] \arrow[rd, "p"] & f^{*}E_{2} \arrow[r, "f"] \arrow[d, "p"] & E_{2} \arrow[d,shift left=1, "p"] \\
            & X_{1} \arrow[r, "f"] \arrow[lu, bend left, "s_{1}"]  & X_{2}  \arrow[u, shift left=1, "s_{2}"]
        \end{tikzcd}
    \end{equation*}
     such that $f$ is equivariant with respect to an algebraic group morphism $G_{1}\to G_{2}$, and $F$ is equivariant for the $G_{1}$-equivariant structure on $f^{*}E_{2}$ induced by $f$. 
     Consider the corresponding morphisms between the LG models 
     \begin{equation*}
         (\on{Tot}(E_{1}^{\vee}), w_{1}) \xlongleftarrow{F^{\on{t}}} (\on{Tot}(f^{*}E_{2}^{\vee}), f^{*}w_{2}) \xlongrightarrow{f} (\on{Tot}(E_{2}^{\vee}), w_{2}).
     \end{equation*}
     Assume $f, F^{\on{t}}$ are proper and the stack morphism $f: [s_{1}^{-1}(0)/G_{1}] \to [s_{2}^{-1}(0)/G_{2}]$ is quasi-smooth, then the functors $f_{*}, f^{*}$ have the following functoriality 
     \begin{equation*}
         \Psi_{2} \circ f_{*} = f_{*}(F^{\on{t}})^{*} \circ \Psi_{1}, \quad \Psi_{1}\circ f^{*} = (F^{\on{t}})_{!}f^{*} \circ \Psi_{2}
     \end{equation*}
     with respect to the Kn{\"o}rrer periodicity $\Psi_{i}: \Db_{G_{i}}(s_{i}^{-1}(0))\simrightarrow \MF_{G_{i}\times \bC^{\times}}(\on{Tot}(E_{i}^{\vee}), w_{i})$. 
\end{lem}

Now we further assume $X$ is a symplectic representation of $G$ and consider the LG model 
\[
w: X\oplus \fg \longrightarrow \bC,
\]
where $w$ is defined from the moment map section $\mu: X \to \fg^{\vee}$ via natural pairing.

\begin{defn}[{\cite[Lemma 2.3]{PT}}] \label{def:winmf}
    For a generic window $\delta+ (1/2)  {\overline{\Sigma}}_{X}$, define the \textit{window subcategory} 
    \[
    \cW_{\delta} \subset \MF_{G\times \bC^{\times}}(X\oplus \fg, w)
    \] 
    as the full subcategory of matrix factorizations whose components are 
    direct sums of 
    equivariant vector bundles $V_{\la} \otimes\cO_{X\oplus \fg}$ for $\la\in \delta+(1/2) {\overline{\Sigma}}_{X}$. 
\end{defn}

\begin{thm}[\cite{HLS}]\label{thm:winmf}
    For a generic window parameter $\delta$ and a generic stability condition $\theta$, the restriction functor $\iota^{*}$ defines an equivalence
    \begin{equation*}
        \iota^{*}: \cW_{\delta} \simrightarrow \MF_{G\times\bC^{\times}}((X\oplus\fg)^{\theta\text{-ss}}, w).
    \end{equation*}
    As a result, given two generic stability conditions $\theta_{1},\theta_{2}$, there is a window equivalence
    \begin{equation*}
        \bW_{\delta}: \MF_{G\times\bC^{\times}}((X\oplus\fg)^{\theta_{1}\text{-ss}}, w)
        \xlongrightarrow{(\iota_{1}^{*})^{-1}} 
        \cW_{\delta} 
        \xlongrightarrow{\,\,\,\,\,\,\iota_{2}^{*}\,\,\,\,\,\,} 
        \MF_{G\times\bC^{\times}}((X\oplus\fg)^{\theta_{2}\text{-ss}},w).
    \end{equation*}    
\end{thm}

For a generic stability condition $\theta$, the critical locus of $w$ has the same intersection with   $(X\oplus\fg)^{\theta\text{-ss}}$ and $X^{\theta\text{-ss}}\times \fg$ \cite[Lemma 5.4]{HLS}. 
Consequently, the derived categories of matrix factorizations over these two loci are equivalent \cite[Lemma 5.5]{HLS}
\begin{equation}\label{eq:critloci}
    \MF_{G\times\bC^{\times}}((X\oplus\fg)^{\theta\text{-ss}}, w) \simeq \MF_{G\times\bC^{\times}}(X^{\theta\text{-ss}}\times \fg, w). 
\end{equation}
Now, if we apply the Kn{\"o}rrer periodicity over the  $\te$-semi-stable loci
\begin{equation*}
    \Psi: \Db(\mu^{-1}(0)^{\theta\text{-ss}}/G) \simrightarrow \mathrm{MF}_{G\times\bC^{\times}}(X^{\theta\text{-ss}}\times \fg, w),
\end{equation*}
and combine it with the identification (\ref{eq:critloci}) as well as the equivalence $\bW_{\delta}$, we obtain a window equivalence for the (smooth) hyperk\"ahler quotient $\mu^{-1}(0)^{\theta\text{-ss}}/G$. 

\begin{cor}[\cite{HLS}]
    For generic parameters $\delta, \theta_{1}, \theta_{2}$ as before, there is a window equivalence
    \begin{equation*}
        \Psi^{-1} \bW_{\delta} \Psi: \Db(\mu^{-1}(0)^{\theta_{1}\text{-ss}}/G) \simrightarrow \cW_{\delta} \simrightarrow \Db(\mu^{-1}(0)^{\theta_{2}\text{-ss}}/G). 
    \end{equation*}    
\end{cor}

\subsection{Lascoux resolutions}
Let $G$ be a reductive group acting on a vector space $T$, and $U\subset T$ a subspace preserved by a parabolic subgroup $P$ of $G$.
For a representation $V$ of $P$, consider the vector bundle $\cV:= G \times_{P} V$ over $G/P$ and denote its pullback to the total spaces $U\times_{P} G \subset T\times G/P$ by the same character. 
Consider 
\begin{equation*}
    \pi: U\times_{P}G \longrightarrow T, \quad (u,g)\longmapsto gu, 
\end{equation*}
which factors through 
\begin{equation*}
	\begin{tikzcd}
		U\times_{P}G \arrow[r, hook, "i"] & T \times G/P \arrow[r, "q"] & T,
	\end{tikzcd}
\end{equation*}
where $i(u,g) = (gu, gP)$ and $q$ is the projection. 
The following theorem provides a Lascoux type resolution of the direct image $\pi_{*}\cV$ (see also \cite[\S 6.1]{Wey}). 

\begin{thm}[{\cite[Theorem A.16]{DS}}]\label{thm:torsion}
    Suppose $\pi_{*}\cV$ is a sheaf concentrated in a single degree, then it has a $G$-equivariant resolution $\cF^{\bullet}$ where
    \begin{equation*}
        \cF^{-n} = \bigoplus_{r \geq n} R^{r -n}q_{*} \left(\cV \otimes \extp^{r} (T/U)^{\vee} \right).
    \end{equation*}
\end{thm}

Note that when $U\subset T$ are vector bundles over a scheme and $\pi$ is constant over that base, we can also apply this theorem relatively.

\section{The categorical action}
We continue with the notations introduced in \S\ref{sec:Nak} and Example~\ref{eg:window} and assume $2k \leq N$.  
Let $V_{k}$ be the rank $k$ tautological vector bundle on various quotients of $X_{k}$. 
A prime mark $V_{k}^{\prime}$ is used when the locus is defined by the stability condition $\theta_{-}$, for example, over the quotient $[X^{-}_{k}\times \gl_{k}/G_{k}]$ or $\TGr{\bC^{N}}{k}$.
Given a dominant weight $\la = (\la_{i})_{i=1}^{k}$, there is a vector bundle $\bS^{\la}V_{k}$ given by the Schur functor construction (see, e.g., \cite{Wey}).
In particular, the construction coincides with the $G_{k}$-equivariant vector bundle $V_{\la}\otimes \cO_{X_{k} \oplus \gl_{k}}$ over the affine space $X_{k}\oplus \gl_{k}$. 

The multiplicative  group $\bC^{\times}$ acts by squared dilation on $\gl_{k}$ and trivially on $X_{k}$. 
The Kn{\"o}rrer periodicity reads as
\begin{equation*}
    \Psi_{k}: \Db(\mu^{-1}(0)^{\theta_{\pm}\text{-ss}}/G_{k}) \simrightarrow \MF_{G_{k}\times \bC^{\times}}(X_{k}^{\pm}\times \gl_{k},w_{k}). 
\end{equation*}

\subsection{Hecke correspondences}\label{sec:Hecke}
For $0\leq m \leq n \leq N$, let
\begin{equation*}
    Z_{m,n} \subset \Hom(V_{m}, V_{n}) \oplus \Hom(V_{n}, \bC^{N}) \oplus \Hom(\bC^{N}, V_{m})
\end{equation*}
be the open subspace where the first two maps give embeddings $V_{m}\hookrightarrow V_{n} \hookrightarrow \bC^{N}$. 
Suppose $G_{m}\times G_{n}$ acts on this vector space of linear maps by change of basis, then the quotient $ Z_{m,n}/(G_{m}\times G_{n})$ is identified with the total space of the Hom bundle $\Hom(\bC^{N},V_{m})$ over the partial flag variety $\Fl(m,n;\bC^{N})$. 
The \textit{Hecke correspondence} is the closed subvariety 
\begin{equation*}
    \fZ_{m,n} \subset Z_{m,n}/(G_{m}\times G_{n})
\end{equation*} 
where $V_{n}$ is contained in the kernel of the tautological map $\bC^{N}\to V_{m}$. 
It is equipped with natural projections
\begin{equation*} 
    \TGr{m}{\bC^{N}}  \xlongleftarrow{\pi_{m}} \fZ_{m,n} \xlongrightarrow{\pi_{n}} \TGr{n}{\bC^{N}}, 
\end{equation*}
where we compose $V_{m}\to V_{n} \to \bC^{N}$ or $\bC^{N}\to V_{m} \to V_{n}$ and forget the relevant group action respectively. 
This also embeds the correspondence $\fZ_{m,n}$ into the product $\TGr{m}{\bC^{N}}\times \TGr{n}{\bC^{N}}$ as it was originally considered in \cite{Nak94,Nak98}. 

By definition, $\fZ_{m,n}$ is the vanishing locus of the composition of the tautological maps $V_{n}\hookrightarrow \bC^{N}\to V_{m}$. 
This defines a regular section to the bundle $\Hom(V_{n},V_{m})$ over $Z_{m,n}$, so we have an LG model
\begin{equation*}
    w_{m,n}: Z_{m,n}\times \Hom(V_{n},V_{m})^{\vee} \longrightarrow \bC. 
\end{equation*}
The Kn{\"o}rrer periodicity for the Hecke correspondence $\fZ_{m,n}$ then reads as 
\begin{equation*}
    \Psi_{m,n}: \Db(\fZ_{m,n}) \simrightarrow \MF_{G_{m}\times G_{n}\times \bC^{\times}}(Z_{m,n}\times \Hom(V_{n},V_{m})^{\vee}, w_{m,n}). 
\end{equation*}
The projections $\pi_{m},\pi_{n}$ also lift to equivariant maps $\Pi_{m},\Pi_{n}$ between the bundles
\begin{equation*}
    X_{m}^{+}\times \gl_{m}^{\vee} \xleftarrow{\pi_{m}} Z_{m,n}\times \gl_{m}^{\vee} \xleftarrow{\Pi_{m}}  Z_{m,n}\times \Hom(V_{n},V_{m}) \xrightarrow{\Pi_{n}} Z_{m,n}\times \gl_{n}^{\vee} \xrightarrow{\pi_{n}} X_{n}^{+}\times \gl_{n}^{\vee},
\end{equation*}
where we pre-compose or post-compose the fibres in $\Hom(V_{n},V_{m})$ with the tautological map $V_{m}\hookrightarrow V_{n}$.
The maps $\Pi_{m},\Pi_{n}$ are compatible with the regular sections, so we are in the setting of Lemma~\ref{lem:fun} and we have the corresponding transposed maps between the LG models
\begin{equation}\label{eq:LGmap}
    Z_{m,n}\times \gl_{m} \xlongrightarrow{\Pi_{m}^{\on{t}}} Z_{m,n} \times \Hom(V_{n},V_{m})^{\vee} \xlongleftarrow{\Pi_{n}^{\on{t}}} Z_{m,n} \times \gl_{n}.
\end{equation}
The morphisms $\Pi_{m}^{\on{t}}$, $\Pi_{n}^{\on{t}}$, which send an element of $\gl_{m}$ or $\gl_{n}$ to its obvious composition with the tautological map $V_{m}\hookrightarrow V_{n}$, are both flat. 
Moreover, the potentials are all compatible in the following way
$$\pi_{m}^{*} w_{m} = (\Pi_{m}^{\on{t}})^{*} w_{m,n}, \quad (\Pi_{n}^{\on{t}})^{*} w_{m,n} = \pi_{n}^{*} w_{n}. $$

For $0\leq i \leq k$, we will also consider the base change of $\fZ_{k-i,N-k}$ along the LMN isomorphism (\ref{eq:LMN})
\begin{equation}\label{eq:transport}
    \begin{tikzcd}
        \fZ_{k-i,k}^{\prime} \arrow[r,"\sim"]  \arrow[d,"\pi_{k}^{\prime}"] \arrow[dr, phantom, "\square"]
        & \fZ_{k-i,N-k} \arrow[d,"\pi_{N-k}"] \\
        \TGr{\bC^{N}}{k} \arrow[r,"\sim"] 
        & \TGr{N-k}{\bC^{N}}. 
    \end{tikzcd}
\end{equation}
Let
\[
Z_{k-i,k}^{\prime} \subset \Hom(V_{k-i},\bC^{N})\oplus\Hom(\bC^{N},V_{k}^{\prime})\oplus\Hom(V_{k}^{\prime},V_{k-i})
\]
be the open subspace where the first two maps are of full rank $V_{k-i} \hookrightarrow \bC^{N} \twoheadrightarrow V_{k}^{\prime}$.
Then, the correspondence 
\begin{equation*}
    \TGr{k-i}{\bC^{N}}  \xlongleftarrow{\pi_{k-i}} \fZ_{k-i,k}^{\prime} \xlongrightarrow{\pi_{k}^{\prime}} \TGr{\bC^{N}}{k} 
\end{equation*}
is the local complete intersection in $Z_{k-i,k}^{\prime}/(G_{k-i}\times G_{k})$ where the composition of $V_{k-i}\hookrightarrow\bC^{N}\twoheadrightarrow V_{k}^{\prime}$ vanishes.  
Again, this defines an LG model
\begin{equation*}
    w_{k-i,k}^{\prime}: Z_{k-i,k}^{\prime} \times \Hom(V_{k-i},V_{k}^{\prime})^{\vee} \longrightarrow \bC, 
\end{equation*}
and the Kn{\"o}rrer periodicity is given by
\begin{equation*}
    \Db(\fZ_{k-i,k}^{\prime}) \simrightarrow \MF_{G_{k-i}\times G_{k}\times \bC^{\times}}(Z_{k-i,k}^{\prime} \times \Hom(V_{k-i},V_{k}^{\prime})^{\vee}, w_{k-i,k}^{\prime}). 
\end{equation*}
Analogous to (\ref{eq:LGmap}), the projections $\pi_{k-i}$, $\pi_{k}^{\prime}$ lift to compatible maps $\Pi_{k-i}$, $\Pi_{k}^{\prime}$ between the Hom bundles, as well as their transposes between the LG models
\begin{equation}\label{eq:LGmap1}
    Z_{k-i,k}^{\prime} \times \gl_{k-i} \xlongrightarrow{\Pi_{k-i}^{\on{t}}} Z_{k-i,k}^{\prime} \times \Hom(V_{k-i},V_{k}^{\prime})^{\vee} \xlongleftarrow{(\Pi_{k}^{\prime})^{\on{t}}} Z_{k-i,k}^{\prime} \times \gl_{k}.
\end{equation}
Here the maps are defined by composing endomorphisms in $\gl_{k-i}$ or $\gl_{k}$ with the tautological map $V_{k}^{\prime} \to V_{k-i}$ in obvious ways. 
Note that because this tautological map can be of arbitrary rank, the morphisms $\Pi_{k-i}^{\on{t}}$, $(\Pi_{k}^{\prime})^{\on{t}}$ are in general not flat.

\subsection{The geometric categorical action}\label{sec:sl2}
Influenced by Nakajima's seminal work \cite{Nak94, Nak98}, it is natural to think of the quiver varieties $\TGr{k}{\bC^{N}}$ for $k=0,1,\cdots, N$ all together as a geometric realization of the weight spaces of an irreducible representation of (the quantized) $\fsl(2)$. 
The following functors are introduced by Cautis, Kamnitzer and Licata to categorify this geometric action. 

\begin{defn}[\cite{CKL-Duke,CKL-sl2}]\label{def:ef}
    For $0\leq m\leq n\leq N$, let
    \begin{equation*}
        \begin{tikzcd}        
        \be^{n,m}: \Db(\TGr{n}{\bC^{N}}) \arrow[r,shift left] & \Db(\TGr{m}{\bC^{N}}): \bbf^{m,n}  \arrow[l,shift left]
        \end{tikzcd}
    \end{equation*}
    be the integral functors induced by the kernels
    \begin{align*}
        \cO_{\fZ_{m,n}} &\otimes \det(\bC^{N}/V_{n})^{m-n} \otimes \det(V_{m})^{n-m},\\
        \cO_{\fZ_{m,n}} &\otimes \det(V_{n}/V_{m})^{N-m-n}
    \end{align*}
    respectively. 
    The functors are also denoted by $\be^{(n-m)}$, $\bbf^{(n-m)}$ when $n,m$ are obvious from the context.
\end{defn}
These functors satisfy the axioms of a strong categorical $\fsl(2)$ action \cite{CKL-Duke}. 
For example, up to degree shifts, the functors $\be,\bbf$ are both left and right adjoints of each other
\begin{equation*}
     \bbf^{m,n}[-(n-m)(N-n-m)] \dashv \be^{n,m}  \dashv \bbf^{m,n}[(n-m)(N-n-m)]. 
\end{equation*} 
In particular, we have the counit  
\begin{equation*}     
    \varepsilon: \bbf^{m,n}\be^{n,m}[-(n-m)(N-n-m)] \longrightarrow \on{id}.    
\end{equation*}
Moreover, powers of $\be$ or $\bbf$ decompose in the following way: for $0\leq l\leq m\leq n\leq N$, 
\begin{align*}
    \be^{m,l} \be^{n,m} &\simeq \be^{n,l}\otimes H^{*}(\Gr(n-m, n-l)), \\
    \bbf^{m,n} \bbf^{l,m} &\simeq \bbf^{l,n}\otimes H^{*}(\Gr(m-l, n-l)). 
\end{align*}
Here we always shift the degree so that the cohomology ring is symmetric with respect
to the 0th degree. For example,  
\begin{equation*}
    H^{*}(\Gr(2,4)) = \bC[4] \oplus \bC[2] \oplus \bC^{2}[0] \oplus \bC[-2] \oplus \bC[-4]. 
\end{equation*}
 
For each $0\leq i\leq k$, we consider the functors 
\begin{equation*}
    \be^{(i)} = \be^{k,k-i},\quad \bbf^{(N-2k+i)} = \bbf^{k-i,N-k},
\end{equation*}
and their composition
\begin{equation*}
    \Theta^{(i)} = \bbf^{(N-2k+i)} \be^{(i)}: \Db(\TGr{k}{\bC^{N}}) \longrightarrow \Db(\TGr{N-k}{\bC^{N}}).
\end{equation*}
A differential map
\begin{equation*}
    d_{i}: \Theta^{(i)}[-i]\longrightarrow \Theta^{(i-1)}[-(i-1)] 
\end{equation*}
can be defined in the following way
\[
\bbf^{(N-2n+i)} \be^{(i)}\langle 1 \rangle [-1] \longrightarrow \bbf^{(N-2n+i-1)} \bbf^{(1)}  \be^{(1)} \be^{(i-1)} \xlongrightarrow{\varepsilon}  \bbf^{(N-2n+i-1)} \be^{(i-1)}.
\] 
Here we first include $\be^{(i)}$ or $\bbf^{(N-2n+i)}$ into the lowest degree in the decomposition
\begin{align*}
	\be^{(1)}\be^{(i-1)} &\simeq \be^{(i)}\otimes H^{*}(\Gr(i-1,i)),\\
	\bbf^{(N-2n+i-1)} \bbf^{(1)} &\simeq \bbf^{(N-2n+i)} \otimes H^{*}(\Gr(1,N-2n+i))
\end{align*}
respectively. 
Next, we apply the counit to the two middle terms
\begin{equation*}
	\varepsilon: \bbf^{(1)}\be^{(1)} \longrightarrow \on{id}[N-2n+2i-1] \langle -N+2n -2i +1 \rangle. 
\end{equation*}
With these differentials, we form the \textit{Rickard complex} of functors
\begin{equation}\label{eq:Rickard}
    \Theta= \left\{\Theta^{(k)}[-k]\xrightarrow{d_{k}} \Theta^{(k-1)}[-(k-1)]\to \cdots \to \Theta^{(1)}[-1] \xrightarrow{d_{1}} \Theta^{(0)}\right\}. 
\end{equation}

Recall that a \textit{Postnikov system} of the complex $\Theta$ is a collection of distinguished triangles
\begin{equation} \label{eq:Pos}
    T^{(1)}[-1] \xlongrightarrow{a_{1}} \Theta^{(1)}[-1] \xrightarrow{d_{1}} \Theta^{(0)}, \quad
    T^{(i)}[-i] \xlongrightarrow{a_{i}} \Theta^{(i)}[-i] \xlongrightarrow{b_{i}} T^{(i-1)}[-i+1]
\end{equation}
such that $a_{i-1} \circ b_{i} = d_{i}$ for $2\leq i \leq k$. 
The object $T^{(k)}$ is called a (right) \textit{convolution} of the complex. 

\begin{thm}[\cite{CKL-sl2}]\label{thm:CKL}
    The Rickard complex has a unique convolution, which defines a natural equivalence 
    \begin{equation*}
        \mathbb{T} = \on{Conv}(\Theta): \Db(\TGr{k}{\bC^{N}}) \simrightarrow \Db(\TGr{N-k}{\bC^{N}})
    \end{equation*}
    for the local model of type $A$ stratified Mukai flops. 
\end{thm}

We can transport the kernel of $\bbf^{(N-2k+i)}$ to $\fZ_{k-i,k}^{\prime}$ via the base change (\ref{eq:transport}). 
The isomorphic integral functor 
\begin{equation*}
    \bbf^{\{ i\}}: \Db(\TGr{k-i}{\bC^{N}}) \longrightarrow \Db(\TGr{\bC^{N}}{k})
\end{equation*}
is then induced by the transported kernel 
\begin{equation*}
    \cO_{\fZ_{k-i,k}^{\prime}}\otimes \det(\bC^{N}/V_{k-i})^{i} \otimes \det(V_{k}^{\prime})^{-i}
\end{equation*}
via the canonical isomorphism $\bC^{N}/V_{N-k} \cong V_{k}^{\prime}$.
The Rickard complex also transports to an isomorphic complex 
\begin{equation}\label{eq:Rickard'}
    \Theta^{\prime} = \left\{
    \Theta^{\{k\}}[-k]\xrightarrow{d_{k}} \Theta^{\{k-1\}}[-(k-1)]\to \cdots \to \Theta^{\{1\}}[-1] \xrightarrow{d_{1}} \Theta^{\{0 \}}
    \right\}  
\end{equation}
with $\Theta^{\{i\}} = \bbf^{\{ i\}} \be^{(i)}$,
which gives rise to an equivalence 
\begin{equation*}
    \bT^{\prime}= \on{Conv}(\Theta^{\prime}): \Db(\TGr{k}{\bC^{N}}) \simrightarrow \Db(\TGr{\bC^{N}}{k}) 
 \end{equation*}
by Theorem~\ref{thm:CKL}.

\subsection{Main results}\label{sec:main}
From now on, we fix a generic window parameter $\delta\in \bR$ such that the window (see Example~\ref{eg:window}) is given by the half-open interval 
\begin{equation*}
    \left[ \lceil \delta - (N/2) \rceil, \lceil \delta + (N/2) \rceil \right) = [-k, N-k). 
\end{equation*}
We may also replace $\delta$ with this half-open interval in the notation of the window subcategory $\cW_{[-k,N-k)}$ or the window equivalence $\bW_{[-k,N-k)}$. 

\begin{thm}\label{thm:main}
    The following diagram of triangulated category equivalences commutes
    \begin{equation*}
    \begin{tikzcd}
        \MF_{G_{k} \times \bC^{\times}}(X_{k}^{+} \times \gl_{k}, w_{k}) \arrow[r, "\bW_{\delta}"] & 
        \MF_{G_{k} \times \bC^{\times}}(X_{k}^{-} \times \gl_{k}, w_{k})  \\
        \Db(\TGr{k}{\bC^{N}}) \arrow[r, "\bT^{\prime}"] \arrow[u, "\Psi_{k}"] & 
        \Db(\TGr{\bC^{N}}{k}) \arrow[u, "\Psi_{k}"].
    \end{tikzcd}
\end{equation*}
\end{thm}

\begin{rmk}[$k=1$] \label{rmk:k=1}
    The abelian case is certainly known to experts, though it has not been written down in the literature.
    In \cite{Hara}, Hara calculated what is equivalent to that $\bT^{\prime}$ maps $\cO(-n) = V_{1}^{\otimes n}$ to $\cO(-n) = (V_{1}^{\prime})^{\otimes n}$ for $n\in[-1,N-1)$, and showed that the direct sum of these $N$ line bundles is a tilting generator of $\Db(T^{*}\bP^{N-1})$. 
    To see the theorem is true in this case, it suffices to verify 
    $$\bW_{[-1,N-1)}(\Psi_{1}(V_{1}^{\otimes n})) = \Psi_{1}((V_{1}^{\prime})^{\otimes n}), \quad n \in [-1,N-1).$$
    Note that the Koszul factorization (\ref{eq:Kmf}) in this case is simply $\extp^{\bullet}\gl_{1}: \cO \rightleftarrows \cO$, so the image $\Psi_{1}(V_{1}^{\otimes n})$ is the matrix factorization $V_{1}^{\otimes n} \otimes  \extp^{\bullet}\gl_{1}: V_{1}^{\otimes n} \rightleftarrows V_{1}^{\otimes n}$ and same for $\Psi_{1}((V_{1}^{\prime})^{\otimes n})$. 
    In other words, the images $\Psi_{1}(V_{1}^{\otimes n})$ and $\Psi_{1}((V_{1}^{\prime})^{\otimes n})$ are restrictions of the common matrix factorization $V_{1}^{\otimes n}\otimes \extp^{\bullet}\gl_{1}$ from the linear stack $[X_{1}\oplus \gl_{1}/G_{1}]$.
    As the highest weights $n$ here are all contained in the window $[-1,N-1)$, this matrix factorization obviously lies in the window subcategory. 
\end{rmk}

The remainder of the paper is devoted to proving  Theorem~\ref{thm:main}. 
Before that, we state two immediate corollaries of the main theorem.

\begin{cor}\label{cor:sph}
    The square of $\bT$ is a window shift autoequivalence, i.e.\ 
    \[
    \Psi_{k} \bT \circ \bT \Psi^{-1}_{k} \simeq \bW_{\delta+1}^{-1}\circ \bW_{\delta}.
    \]
\end{cor}

From \cite{HLShip}, we know that such a window shift autoequivalence is given by spherical twists around certain spherical functors arising from the Kempf--Ness stratification of the LG model $X_{k}\oplus \gl_{k}$. 
Spelling out these spherical twists will be of independent interest.

\begin{proof}
    By definition, the flop equivalence in the opposite direction 
    \begin{equation*}
        \bT_{N-k}: \Db(\TGr{N-k}{\bC^{N}}) \simrightarrow \Db(\TGr{k}{\bC^{N}})
    \end{equation*}
    is given by the convolution of a similar Rickard complex 
    \begin{equation*}
    \left\{\bbf^{(k)}\be^{(N-k)}[-k]\to \cdots \to \bbf^{(1)}\be^{(N-2k+1)}[-1] \to \be^{(N-2k)} \right\},
    \end{equation*}
    cf.\ the Rickard complex (\ref{eq:Rickard}) of $\bT_{k}$. 
    By \cite[Lemma 7.4]{CKL}, the kernel of $\bT_{N-k}$ is isomorphic to that of $\bT_{k}$ tensoring with $L^{N-2k}$, where $L$ is the line bundle 
    \[
    \det(V_{k})\otimes \det(V_{N-k}) \otimes \det(\bC^{N})^{-1}.
    \]
    On the other hand, the inverse 
    \begin{equation*}
        \bT_{k}^{-1}: \Db(\TGr{N-k}{\bC^{N}}) \simrightarrow \Db(\TGr{k}{\bC^{N}})
    \end{equation*}
    is given by the convolution of the complex of adjoints
    \begin{equation*}
    \left\{\be^{(N-2k)} \to \bbf^{(1)}\be^{(N-2k+1)}[1] \to \cdots \to \bbf^{(k)}\be^{(N-k)}[k] \right\}.
    \end{equation*}
    From \textit{loc.\ cit.}, the kernel of $\bT_{k}^{-1}$ is isomorphic to that of $\bT_{k}$ tensoring with $L^{N-2k-1}$. 
    Thus, we have
    \begin{equation*}
        \bT_{N-k} \simeq (-) \otimes \det(V_{k}) \circ \bT_{k}^{-1} \circ  (-) \otimes \det(\bC^{N}/V_{N-k})^{-1}.
    \end{equation*}
    Combining this with Theorem~\ref{thm:main}, we are able to express $\bT^{2}$ as a composition of window equivalences
    \begin{align*}
        \Psi \bT_{N-k}\circ \bT_{k} \Psi^{-1}  &\simeq (-) \otimes \det(V_{k}) \circ \bW_{\delta}^{-1} \circ  (-) \otimes \det(V_{k}^{\prime})^{-1} \circ \bW_{\delta} \\
        &  \simeq \bW_{\delta+1}^{-1}\circ \bW_{\delta}.
    \end{align*}
\end{proof}

\begin{cor}\label{cor:inv}
    Let $\la = (\la_{i})_{i=1}^{k}$ be a dominant weight such that $-1\leq \la_{i}\leq N-2k$, then $\bT$ maps the Schur functor $\bS^{\la}V_{k}$ to the Schur functor $\bS^{\la}(\bC^{N}/V_{N-k})$.   
\end{cor}

In other words, such a Schur functor is invariant under the GIT wall-crossing 
\[
\bT^{\prime} (\bS^{\la}V_{k}) = \bS^{\la} V_{k}^{\prime}. 
\]
When $k=1$, these Schur functors coincide with a choice of Beilinson's collection $\{ \cO(2-N)$, $\cdots$, $\cO(1)\}$, and as mentioned in Remark~\ref{rmk:k=1}, the invariance is already known in \cite{Hara}.
For $k\geq 2$, the set of Schur functors in this corollary is a proper subset of (any choice of) Kapranov's collection \cite{Kap} (that contains it). 
In \cite{Tseu}, we constructed a tilting generator of $\Db(\TGr{2}{\bC^{N}})$, whose direct summands are given by these invariant Schur functors along with other extension bundles taken to replace the rest of Kapranov's collection.

\begin{proof}
    The equivariant bundle $\bS^{\la}V_{k}$ on $X_{k}$ restricts to the same Schur functor $\bS^{\la}V_{k}$ or $\bS^{\la}V_{k}^{\prime}$ over the semi-stable loci $X_{k}^{\pm}$. 
    Under the Kn{\"o}rrer periodicity 
    \begin{equation*}
        \Psi_{k}: \Db([\mu^{-1}(0)/G_{k}]) \simrightarrow \MF_{G_{k}\times \bC^{\times}}(X_{k}\oplus \gl_{k}, w_{k}), 
    \end{equation*}
    this bundle $\bS^{\la}V_{k}$ is mapped to $\bS^{\la}V_{k} \otimes \extp^{\bullet}\gl_{k}$, its tensor product with the Koszul factorization (\ref{eq:Kmf}). 
    We shall verify  that all the Schur functors $\bS^{\nu}V_{k}$ appearing in a decomposition of this tensor product have highest weights located in the window $[-k,N-k)$.
    In particular, this implies that the matrix factorization $\bS^{\la}V_{k} \otimes \extp^{\bullet}\gl_{k}$ lies in the window subcategory.
    As a result, its restriction $\bS^{\la}V_{k} \otimes \extp^{\bullet}\gl_{k} = \Psi_{k}(\bS^{\la}V_{k})$ to $X_{k}^{+}\times \gl_{k}$ is sent under the window equivalence to its restriction $\bS^{\la}V_{k}^{\prime} \otimes \extp^{\bullet}\gl_{k} = \Psi_{k}(\bS^{\la}V_{k}^{\prime})$ to $X_{k}^{-}\times \gl_{k}$.
    
    We first decompose each exterior power $\bigwedge^{r} \gl_{k}$ as 
    \begin{equation*}
        \extp^{r} (V_{k}^{\vee}\otimes V_{k} ) \cong \bigoplus_{\alpha} \bS^{\alpha} V_{k}^{\vee} \otimes \bS^{\alpha^{\prime}} V_{k},
    \end{equation*}
    where the direct sum is over all Young diagrams $\alpha = (\alpha_{i})_{i=1}^{k}$ with $\alpha_{i} \in [0,k]$ and $\sum_{i} \alpha_{i} = r$. 
    Here the conjugate Young diagram $\alpha^{\prime} = (\alpha_{i}^{\prime})_{i=1}^{k}$ is defined by counting the number of boxes in each column of $\alpha$. 
    The summand further decomposes 
    \begin{equation*}
        \bS^{\alpha} V_{k}^{\vee} \otimes \bS^{\alpha^{\prime}} V_{k} \cong \bS^{-\alpha_{k},\cdots, -\alpha_{1}} V_{k} \otimes \bS^{\alpha^{\prime}} V_{k} \cong \bigoplus_{\mu} c_{-\alpha,\alpha^{\prime}}^{\mu} \bS^{\mu} V_{k}
    \end{equation*}
    by the Littlewood--Richardson rule \cite[\S 2.3]{Wey}. 
    We claim that each $\mu=(\mu_{i})_{i=1}^{k}$ satisfies $\mu_{1}\leq k-1$ and $\mu_{k}\geq 1-k$. 
    These two conditions are obvious when $\alpha_{1}^{\prime} \leq k-1$ and $\alpha_{1} \leq k-1$ respectively. 
    Assume the first column of $\alpha$ has $k$ boxes, then we immediately have $\alpha_{k}\geq 1$, hence $\mu_{1}\leq k-1$.
    On the other hand, $\alpha_{1} = k$ implies $\alpha_{k}^{\prime}\geq 1$, so in this case we also have $\mu_{k}\geq 1-k$. 
    By the Littlewood--Richardson rule again, each tensor product $\bS^{\mu}V_{k} \otimes \bS^{\la} V_{k}$ decomposes into a direct sum of Schur functors $\bS^{\nu} V_{k}$, where $\nu _{1}\leq N-k+1$ and $\nu_{k}\geq -k$ as expected. 
\end{proof}

\subsection{The kernels}\label{sec:kernel}
In this subsection, we find the matrix factorization kernels that induce the twisted functors $\Psi \be \Psi^{-1}$, $\Psi \bbf \Psi^{-1}$, and $\Psi \Theta \Psi^{-1}$ of the categorical $\fsl(2)$ action. 
The procedure is partly inspired by \cite[\S 3]{Toda-flip}. 

By definition (see, e.g., \cite{BFK-2,Hira-2}), an integral functor $\Phi_{F}$ between the derived categories of matrix factorizations of two gauged LG models $(Y_{i},G_{i},w_{i})$, $i=1,2$ is induced by a kernel 
\[
F \in \MF_{G_{1}\times G_{2}\times \bC^{\times}}(Y_{1}\times Y_{2}, p_{2}^{*} w_{2} - p_{1}^{*} w_{1})
\]
via the Fourier--Mukai type formula $\Phi_{F}(-) = p_{2,*}(p_{1}^{*}(-)\otimes F)$, where $p_{i}:Y_{1}\times Y_{2} \to Y_{i}$, $i=1,2$ are the projections.

Consider the fibre product of (\ref{eq:LGmap})
\begin{equation*}
    \begin{tikzcd}
        W_{m,n} \arrow[r, "p_{n}"] \arrow[d, "p_{m}"] \arrow[dr, phantom, "\square"]
        & Z_{m,n}\times \gl_{n} \arrow[d, "\Pi_{n}^{\on{t}}"] \\
        Z_{m,n}\times \gl_{m} \arrow[r, "\Pi_{m}^{\on{t}}"]
        & Z_{m,n}\times \Hom(V_{m},V_{n}).
    \end{tikzcd}
\end{equation*}
This is the locally closed subvariety in $Z_{m,n}\times \gl_{m} \times \gl_{n}$ where the map $\alpha_{m,n}: V_{m}\hookrightarrow V_{n}$ from $Z_{m,n}$ intertwines with the endomorphisms $\epsilon_{m} \in \gl_{m}$ and $\epsilon_{n} \in \gl_{n}$, i.e.\ 
\begin{equation} \label{eq:intcond1}
    \alpha_{m,n} \epsilon_{m} = \epsilon_{n} \alpha_{m,n}.
\end{equation}
In other words, $\epsilon_{n}$ restricts to $\epsilon_{m}$. 
As $Z_{m,n}$ is naturally embedded into $X_{m}^{+} \times X_{n}^{+}$, the fibre product $W_{m,n}$ is a correspondence with natural projections
\begin{equation*}
    X_{m}^{+}\times \gl_{m} \xlongleftarrow{\pi_{m} p_{m}} W_{m,n} \xlongrightarrow{\pi_{n} p_{n}} X_{n}^{+}\times \gl_{n}. 
\end{equation*}
The quotient $[W_{m,n}/G_{m}\times G_{n}]$ is also known as a Hecke correspondence of `triple quiver varieties' \cite[\S 3]{VV}.  

Similarly, consider the fibre product of (\ref{eq:LGmap1})
\begin{equation}\label{eq:torind}
    \begin{tikzcd}
        W_{k-i,k}^{\prime} \arrow[r, "p_{k}^{\prime}"] \arrow[d, "p_{k-i}"] \arrow[dr, phantom, "\square"]
        & Z_{k-i,k}^{\prime} \times  \gl_{k} \arrow[d, "(\Pi_{k}^{\prime})^{\on{t}}"] \\
        Z_{k-i,k}^{\prime} \times  \gl_{k-i} \arrow[r, "\Pi_{k-i}^{\on{t}}"]
        & Z_{k-i,k}^{\prime} \times \Hom(V_{k}^{\prime}, V_{k-i}).
    \end{tikzcd}
\end{equation}
The correspondence $W_{k-i,k}^{\prime}$ is identified with the locus from $Z_{k-i,k}^{\prime} \times \gl_{k-i} \times \gl_{k}$ where the map $\beta_{k-i}: V_{k}^{\prime} \to V_{k-i}$ in $Z_{k-i,k}^{\prime}$ intertwines with the endomorphisms $\epsilon_{k-i} \in \gl_{k-i}$ and $\epsilon_{k}^{\prime} \in \gl_{k}$, i.e.\ 
\begin{equation}\label{eq:intcond2}
    \beta_{k-i} \epsilon_{k}^{\prime} = \epsilon_{k-i} \beta_{k-i}.
\end{equation}
This is a condition that defines a complete intersection, and it is straightforward to check that the fibre product (\ref{eq:torind}) is also Tor-independent.

\begin{lem}\phantomsection \label{lem:ef}
    \begin{enumerate}[wide, labelindent=0pt, label = (\arabic*)]
        \item The functor $\Psi_{m} \be^{n,m} \Psi_{n}^{-1}$ is induced by the kernel
        \begin{equation*}
            \cO_{W_{m,n}}\otimes \det(V_{m})^{-m}\otimes \det(V_{n})^{n} \otimes \det(\bC^{N})^{m-n} \langle mn-n^{2} \rangle.
        \end{equation*} 
        \item The functor $\Psi_{n} \bbf^{m,n} \Psi_{m}^{-1}$ is induced by the kernel
        \begin{equation*}
            \cO_{W_{m,n}} \otimes \det(V_{m})^{m-N} \otimes \det(V_{n})^{N-n} \langle mn-m^{2} \rangle.
        \end{equation*} 
        \item The functor $\Psi_{k} \bbf^{\{ i\}} \Psi_{k-i}^{-1}$ is induced by the kernel 
        \begin{equation*}
            \cO_{W_{k-i,k}^{\prime}}\otimes \det(V_{k-i})^{k-i}\otimes \det(V_{k}^{\prime})^{-k} \otimes \det(\bC^{N})^{i} \langle ik-i^{2}\rangle. 
        \end{equation*} 
    \end{enumerate}
\end{lem}

\begin{proof}
    By Definition~\ref{def:ef} and Lemma~\ref{lem:fun}, the functor $\Psi_{m} \be^{n,m} \Psi_{n}^{-1}$ is given by
    \begin{align*}
        & \Psi_{m} \pi_{m,*} \circ (-)\otimes \det(\bC^{N}/V_{n})^{m-n}\otimes\det(V_{m})^{n-m} \circ \pi_{n}^{*} \Psi_{n}^{-1}\\
        = & \pi_{m,*} (\Pi_{m}^{\on{t}})^{*} \Psi_{m,n} \circ (-)\otimes \det(\bC^{N}/V_{n})^{m-n}\otimes\det(V_{m})^{n-m} \circ \Psi_{m,n}^{-1} (\Pi_{n}^{\on{t}})_{!} \pi_{n}^{*}. 
    \end{align*}
    The shriek pushforward  $(\Pi_{n}^{\on{t}})_{!} = (\Pi_{n}^{\on{t}})_{*} (- \otimes \omega_{\Pi_{n}^{\on{t}}} [\dim \Pi_{n}^{\on{t}}])$, where
    \begin{equation*}
          \omega_{\Pi_{n}^{\on{t}}} = \det(V_{m})^{-n}\otimes \det(V_{n})^{m} \langle -2n(n-m) \rangle, \quad \dim \Pi_{n}^{\on{t}} = n(n-m). 
    \end{equation*}
    Here the degree shift comes from the squared dilation of $\bC^{\times}$  on $\gl_{n}$ and  $\Hom(V_{m},V_{n})$.
    Tensoring with these determinant line bundles commutes with the direct and inverse image functors, hence it also commutes with the Kn{\"o}rrer periodicity (Remark~\ref{rmk:kno}). 
    Now, by flat base change 
    \begin{equation*}
        \pi_{m,*} (\Pi_{m}^{\on{t}})^{*}(\Pi_{n}^{\on{t}})_{*} \pi_{n}^{*} \simeq \pi_{m,*} p_{m,*} p_{n}^{*} \pi_{n}^{*}
    \end{equation*}
    and the projection formula (see, e.g., \cite[\S 4]{Hira-2}), we see that $\Psi_{m} \be^{n,m} \Psi_{n}^{-1}$ is induced by the asserted kernel on $W_{m,n}$. 
    Note that due to the intertwining condition (\ref{eq:intcond1}), the pullback of $w_{n}$ and $w_{m}$ to $(X_{m}^{+}\times \gl_{m}) \times (X_{n}^{+}\times \gl_{n})$ agrees over $W_{m,n}$.
    This ensures that the structure sheaf of $W_{m,n}$ is a reasonable matrix factorization kernel. 
    For the same reason, the cohomological degree shift $[\cdot]$ and the internal degree shift $\langle \cdot \rangle$ are the same on $\cO_{W_{m,n}}$, and the total degree shift is $-2n(n-m) + n(n-m) = mn-n^{2}$. 
    
    The arguments for the other two statements are the same.
    Note that since (\ref{eq:torind}) is (derived) Cartesian, the derived base change formula $(\Pi_{k}^{\prime})^{\on{t},*} (\Pi_{k-i}^{\on{t}})_{*}  \simeq  p^{\prime}_{k,*} p_{k-i}^{*}$ still holds. 
\end{proof}

Let $p_{12},p_{13},p_{23}$ be the projections that respectively forget the third, second, and  first factor of the triple product 
\begin{equation*}
    (X_{k}^{+}\times \gl_{k}) \times (X_{k-i}^{+}\times \gl_{k-i}) \times (X_{k}^{-} \times \gl_{k}). 
\end{equation*}
The scheme-theoretic intersection $I_{k-i} = p_{12}^{-1}(W_{k-i,k})\cap p_{23}^{-1}(W_{k-i,k}^{\prime})$ is isomorphic to the locally closed subvariety of the affine space
\begin{align*}
    \Hom(V_{k-i},V_{k}) & \oplus \Hom(V_{k},\bC^{N}) \oplus \Hom(\bC^{N},V_{k}^{\prime}) \\ & \oplus \Hom(V_{k}^{\prime},V_{k-i}) \oplus \gl(V_{k-i}) \oplus \gl(V_{k}) \oplus \gl(V_{k}^{\prime}) \notag
\end{align*}
where the maps in the first row are all of full rank and the endomorphisms in the second row satisfy both of the intertwining conditions (\ref{eq:intcond1}) and (\ref{eq:intcond2}). 
Intuitively, the intersection $I_{k-i}$ consists of quiver representations
\begin{equation}\label{eq:quiver}
    \begin{tikzcd}
        V_{k}^{\prime} \arrow[out=120,in=60,loop,looseness=2.7, "\epsilon_{k}^{\prime}"] \arrow[r, "\beta_{k-i}"] 
        & V_{k-i} \arrow[r,hook,"\alpha_{k-i}"] \arrow[out=120,in=60,loop,looseness=2.7, "\epsilon_{k-i}"] 
        & V_{k} \arrow[out=120,in=60,loop,looseness=2.7, "\epsilon_{k}"] \arrow[ld,  hook',"a"]  \\
        & \bC^{N} \arrow[lu,  two heads,"b"] &
    \end{tikzcd}
\end{equation}
with relations $\epsilon_{k-i} = \epsilon_{k}|_{V_{k-i}}$ and $\beta_{k-i} \epsilon_{k}^{\prime} = \epsilon_{k-i} \beta_{k-i}$. 
These two intertwining relations together imply another intertwining condition 
\[
(\alpha_{k-i} \circ \beta_{k-i}) \epsilon_{k}^{\prime} = \epsilon_{k} (\alpha_{k-i} \circ \beta_{k-i}). 
\]
Thus, by forgetting $V_{k-i}$ and composing $\alpha_{k-i} \circ \beta_{k-i}$, the intersection $I_{k-i}$ is mapped to the correspondence $I_{k} = W_{k,k}^{\prime}$, giving a partial resolution of the determinantal locus
\begin{equation*}
    C_{\leq k-i} =  \left\{\on{rk} (\beta_{k}: V_{k}^{\prime}\to V_{k}) \leq k-i \right\} \subseteq I_{k}.
\end{equation*}
This locus consists of $k-i+1$ equidimensional components
\begin{equation*}
    C_{j} = \on{cl} \{ \on{rk} (\beta_{k}) = j \}  \subset I_{k}, \quad j=0,1,\cdots,k-i, 
\end{equation*}
where $\on{cl}\{ \cdot \}$ denotes the closure.

As $\be^{(0)} = \on{id}$, the right end term $\Theta^{\{0\}}$ of the Rickard complex (\ref{eq:Rickard'}) is simply $\bbf^{\{ 0\}}$. 
So, the kernel of $\Psi_{k}\Theta^{\{0\}}\Psi_{k}^{-1}$ is supported on $I_{k} = W_{k,k}^{\prime}$ and given by
\[
\cO_{I_{k}}\otimes \det(V_{k})^{k} \otimes \det(V_{k}^{\prime})^{-k}. 
\]
For $i>0$, $\Theta^{\{i\}} = \bbf^{\{ i\}} \be^{(i)}$, and we need to calculate the convolution of the kernels of $\Psi_{k}\bbf^{\{ i\}} \Psi_{k-i}^{-1}$ and $\Psi_{k-i} \be^{k,k-i} \Psi_{k}^{-1}$.

\begin{prop}\label{prop:Dc}
    For $i\geq 1$, the functor $\Psi_{k} \Theta^{\{i\}} \Psi_{k}^{-1}$ is induced by the sheaf
        \begin{equation*}
            R^{0}p_{13,*}\cO_{I_{k-i}} \otimes \det(V_{k})^{k} \otimes \det(V_{k}^{\prime})^{-k} \langle -i^{2} \rangle. 
        \end{equation*} 
\end{prop}

\begin{proof}
    The determinant bundles and the degree shift can be obtained from Lemma~\ref{lem:ef} directly. 
    It remains to show that 
    \[
    p_{13,*} (p_{12}^{*} \cO_{W_{k-i,k}} \otimes p_{23}^{*} \cO_{W_{k-i,k}^{\prime}}) \cong R^{0}p_{13,*}\cO_{I_{k-i}}.
    \]
    The intersection $I_{k-i}$ is local complete of the expected dimension $2kN$ (after modulo the group action of $G_{k-i}\times G_{k} \times G_{k}$), so 
    \begin{equation*}
        p_{12}^{*} \cO_{W_{k-i,k}} \otimes p_{23}^{*} \cO_{W_{k-i,k}^{\prime}} \cong  \cO_{I_{k-i}}. 
    \end{equation*}
    Now, the canonical bundle of $I_{k-i}$ (modulo the group action) is given by
    \[
    \det(V_{k})^{N}\otimes \det(V_{k}^{\prime})^{-N} \langle -2(k^{2}+i^{2}) \rangle,
    \]
    which can be computed as follows. 
    Firstly, $I_{k-i}$ is cut out by the conditions (\ref{eq:intcond1}) and  (\ref{eq:intcond2}), so the canonical bundle of $I_{k-i}$ is isomorphic to the tensor product of 
    \[\det(V_{k-i}^{\vee}\otimes V_{k})\otimes \det((V_{k}^{\prime})^{\vee}\otimes V_{k-i}) \langle 4k(k-i)\rangle\]
    with the canonical bundle of the total space of $\Hom(V_{k}^{\prime},V_{k-i})\times \gl_{k}\times \gl_{k-i}\times \gl_{k}$ over the partial flag varieties $\Fl(V_{k-i}, V_{k};\bC^{N})\times \Gr(\bC^{N},V_{k}^{\prime})$. 
    The vector bundle part of this total space contributes
    \[\det(V_{k}^{\prime}\otimes V_{k-i}^{\vee}) \langle -4k^{2}-2(k-i)^{2} \rangle\]
    to its canonical bundle. 
    The underlying partial flag variety $\Fl(V_{k-i}, V_{k};\bC^{N})$ can be viewed as a relative $\Gr(V_{k}/V_{k-i}, \bC^{N}/V_{k-i})$-bundle over $\Gr(V_{k-i},\bC^{N})$. 
    Tensoring the above two lines of  determinant line bundles with the canonical bundles of the Grassmannians
    \[
    \omega_{\Gr(V_{k}/V_{k-i}, \bC^{N}/V_{k-i})} \otimes \omega_{\Gr(V_{k-i},\bC^{N})} \otimes \omega_{\Gr(\bC^{N},V_{k}^{\prime})}
    \]
    yields $\omega_{I_{k-i}}$.  
    
    Since this canonical bundle of $I_{k-i}$ is free of $V_{k-i}$, it is the pullback of the same line bundle from $C_{\leq k-i}$. 
    By the Kawamata--Viehweg vanishing theorem (see \cite[p.95]{CKL-sl2}), there is no higher cohomology 
    $p_{13,*} \cO_{I_{k-i}} = {R}^{0} p_{13,*} \cO_{I_{k-i}}$.
\end{proof}

\begin{rmk} 
    Note that the deepest component $C_{0}$  is always smooth and $p_{13}$ is isomorphic over it  $p_{13,*}\cO_{I_{0}} \cong \cO_{C_{0}}$. 
    In general, the components $C_{\leq k-i}$ are not normal, nor are the fibres of $p_{13}$ connected. 
    So, unlike in \cite[Proposition 6.3]{CKL-sl2} where the image is the normalization of a single component (i.e.\ the component $\fZ^{(i)}$ mentioned in Introduction), the direct image $p_{13,*} \cO_{I_{k-i}}$ for $i\neq 0,k$ is not isomorphic to the structure sheaf of the normalization of $C_{\leq k-i}$. 
\end{rmk}

In conclusion, the twisted Rickard complex $\Psi_{k} \Theta^{\prime} \Psi_{k}^{-1}$ is induced by the corresponding complex of kernels (setting $p=p_{13}$)
\begin{equation}\label{eq:Ricker}
    p_{*}\cO_{I_{0}} \langle -k^{2}-k \rangle \xrightarrow{d_{k}}  \cdots \xrightarrow{d_{i+1}} p_{*}\cO_{I_{k-i}} \langle -i^{2}-i \rangle \xrightarrow{d_{i}} \cdots \xrightarrow{d_{1}} \cO_{I_{k}}, 
\end{equation}
up to tensoring with the common line bundle $\det(V_{k})^{k} \otimes \det(V_{k}^{\prime})^{-k}$. 
The differentials $d_{i}$ (by abuse of notation) are transported from the original Rickard complex $\Theta^{\prime}$ via the Kn{\"o}rrer periodicity equivalence $\Psi_{k}$.
We also merge the degree shifts $\langle -i^{2} \rangle [-i]$ since they are the same on these sheaves in the matrix factorization category.

\begin{rmk}[$k=1$] \label{thm:k=1}
    In the abelian case, the cone of this complex (\ref{eq:Ricker}) of kernels can be easily understood. 
    We give a sketch proof that there is an exact triangle 
    \[
    \cO_{C_{0}}\langle -2 \rangle  \xlongrightarrow{d_{1}} \cO_{I_{1}} \longrightarrow \cO_{C_{1}}. 
    \]
    The correspondence $I_{1}$ is cut out by the equation $\beta_{1}(\epsilon_{1} -\epsilon^{\prime}_{1})=0$, consisting of two components $C_{1} =\{\epsilon_{1} = \epsilon^{\prime}_{1} \}$ and $C_{0} = \{ \beta_{1} =0\}$. 
    So, a natural weight two map $\cO_{C_{0}} \to \cO_{I_{1}}$ that completes into the above triangle is given by multiplying $\epsilon_{1} -\epsilon^{\prime}_{1}$ (remember that $\bC^{\times}$ acts by squared dilation on $\epsilon_{1}, \epsilon_{1}^{\prime}$).  
    This map is actually equivalent to $d_{1}$ up to a scalar because  
    \begin{align*}
        \Ext^{0}_{I_{1}}(\cO_{C_{0}}\langle -2 \rangle, \cO_{I_{1}})^{\bC^{\times}} & \cong \Ext^{0}_{C_{0}}(\cO_{C_{0}} \langle -2 \rangle, i^{!} \cO_{I_{1}})^{\bC^{\times}} \\
        & \cong H^{0}_{\bC^{\times}}(C_{0}) \cong \bC,
    \end{align*}
    where the first isomorphism comes from the adjunction $i_{*} \dashv i^{!}$ for the embedding $i: C_{0} \to I_{1}$, and the second line is implied by $H^{0}_{\bC^{\times}}(C_{0}) = H^{0}(\Gr(1,\bC^{N})\times \Gr(\bC^{N},1))$. 
\end{rmk}

This observation for $k=1$, together  with Remark~\ref{rmk:k=1}, leads us to speculate that the correct kernel to induce both the window equivalence $\bW_{[-k,N-k)}$ and the twisted equivalence $\Psi_{k} \bT^{\prime} \Psi^{-1}_{k}$ is simply given by the sheaf 
\[
\cO_{C_{k}} \otimes \det(V_{k})^{k} \otimes \det(V_{k}^{\prime})^{-k}.
\]
For $k=2$, it is not obvious to the author that the convolution of the complex (\ref{eq:Ricker})
is quasi-isomorphic to $\cO_{C_{2}}$ (as matrix factorizations).
However, Segal \cite{Seg} showed that $\cO_{C_{2}}$ indeed induces the window equivalence $\bW_{[0,N)}$ in this case. 
We will discuss this in more detail in the next work. 

It is worth mentioning that the correspondence $I_{k}= W_{k,k}^{\prime}$ is potentially related to the `partial compactification' (\cite{BDF,BDF-2}) of the $G_{k}$-action on the LG model $(X_{k}\oplus \gl_{k},w_{k})$. 
We hope our work offers some insights for further comparison.

\subsection{Extension to stacks}\label{sec:1}
Let $\iota_{\pm}: X_{k}^{\pm}\times \gl_{k} \hookrightarrow X_{k} \oplus \gl_{k}$ be the inclusions.
Given its definition
\[
\bW_{\delta}: \MF_{G_{k} \times \bC^{\times}}(X^{+}_{k} \times \gl_{k}, w_{k})
\xlongrightarrow{(\iota_{+}^{*})^{-1}} 
\cW_{\delta} 
\xlongrightarrow{\,\,\,\,\,\,\iota_{-}^{*}\,\,\,\,\,\,}
\MF_{G_{k} \times\bC^{\times}}(X^{-}_{k} \times \gl_{k}, w_{k}), 
\]
one should expect (see also \cite[\S 2.3]{HL}) the window equivalence to be induced by a kernel of the form $(\on{id}\times \iota_{-})^{*}F$, which is the restriction of a matrix factorization kernel
\[
F \in \MF_{G_{k} \times G_{k} \times \bC^{\times}} ( (X^{+}_{k} \times \gl_{k}) \times (X_{k} \oplus \gl_{k}), p_{2}^{*}w_{k} - p_{1}^{*}w_{k} )
\]
from the larger space $(X^{+}_{k} \times \gl_{k}) \times (X_{k} \oplus \gl_{k})$. 
Moreover, the kernel is characterized by the following two conditions.
\begin{enumerate}[wide, labelindent=0pt, label = (\roman*)]
    \item The kernel restricts to the diagonal $\Delta$ of $X^{+}_{k} \times \gl_{k}$, i.e.\ $(\on{id}\times \iota_{+})^{*} F \cong \cO_{\Delta}$.  
    \item The essential image of the functor $\Phi_{F}$ lies in the window subcategory $\cW_{\delta}$. 
\end{enumerate}

In this subsection, we first describe an obvious extension of the kernels of the twisted complex $\Psi_{k} \Theta^{\prime} \Psi_{k}^{-1}$ to $(X^{+}_{k} \times \gl_{k}) \times (X_{k} \oplus \gl_{k})$.
After that, we give a proof of the main theorem by verifying that the convolution of this extended complex of kernels satisfies the above two characterization conditions.

Intuitively, the extension is obtained by simply relaxing the surjectivity condition over the tautological direction $\bC^{N} \twoheadrightarrow V_{k}^{\prime}$ throughout our previous constructions.
Suppose in the affine space
\begin{align}\label{eq:R}
    \Hom(V_{k},\bC^{N}) & \oplus \Hom(\bC^{N},V_{k}^{\prime})  \oplus \Hom(V_{k-i},V_{k})    \\ & \oplus \Hom(V_{k}^{\prime},V_{k-i}) \oplus \gl(V_{k-i}) \oplus \gl(V_{k}) \oplus \gl(V_{k}^{\prime}) \notag
\end{align}
a typical point is denote by $(a,b,\alpha_{k-i},\beta_{k-i},\epsilon_{k-i},\epsilon_{k},\epsilon_{k}^{\prime})$ in order (cf.\ (\ref{eq:quiver})). 
For each $0\leq i \leq k$, consider the following locally closed subvariety 
of (\ref{eq:R})
\[
\dR_{k-i} = \left\{ a, \alpha_{k-i} \text{ of full rank and } \epsilon_{k}|_{V_{k-i}} = \epsilon_{k-i} \right\}. 
\]
Equip $\dot{R}_{k-i}$ with the potential 
\begin{equation*}
    \upsilon_{k-i}: (a,b,\alpha_{k-i},\beta_{k-i},\epsilon_{k-i},\epsilon_{k},\epsilon_{k}^{\prime})  \longmapsto \on{Tr} \left( ba\alpha_{k-i}(\beta_{k-i}\epsilon_{k}^{\prime}- \epsilon_{k-i}\beta_{k-i}) \right),
\end{equation*}
and consider the following local complete intersections
\[
\dI_{k-i} = \{\beta_{k-i}\epsilon_{k}^{\prime} -  \epsilon_{k-i}\beta_{k-i} = 0\}, \quad \dJ_{k-i} = \{ba\alpha_{k-i}=0\}  
\]
in this LG model. 
As explained in \S\ref{sec:mf}, there is a canonical isomorphism of matrix factorizations 
\begin{equation}\label{eq:xy}
    \cO_{\dI_{k-i}} \cong \cO_{\dJ_{k-i}} \otimes \det(V_{k-i})^{-k} \otimes \det(V_{k}^{\prime})^{k-i} [-k(k-i)]
\end{equation}
in $\MF_{G_{k-i}\times G_{k}^{2} \times \bC^{\times}} ({\dR}_{k-i},\upsilon_{k-i})$ since they have a common Koszul resolution. 

Note that because of the relaxation of the surjectivity condition, $\dI_{k}\subset \dR_{k}$ are no longer correspondences. 
However, the previous embedding is still a well-defined morphism
\begin{align*}
    \pi: \dR_{k} & \longrightarrow (X_{k}^{+}\times \gl_{k}) \times (X_{k}\oplus \gl_{k}), \\ 
    (a,b,\beta_{k},\epsilon_{k},\epsilon_{k}^{\prime}) & \longmapsto ((a,\beta_{k} b, \epsilon_{k}),(a \beta_{k}, b, \epsilon_{k}^{\prime})). 
\end{align*}
By forgetting $V_{k-i}$ and composing $\beta_{k-i}\circ \alpha_{k-i}$, we have a projection (denoted by $p_{13}$ previously)
\[
p=p_{k-i}: \dR_{k-i} \longrightarrow \dR_{k}, \quad p(\dI_{k-i}) \subset \dI_{k}. 
\]
By Proposition~\ref{prop:Dc}, the image $(\pi\circ p)_{*} \cO_{\dI_{k-i}} \langle -i^{2} \rangle$ restricts to the kernel $p_{13,*}\cO_{I_{k-i}}\langle -i^{2} \rangle$ of $\Psi_{k}\Theta^{\{i\}} \Psi_{k}^{-1}$  along the inclusion $\on{id}\times \iota_{-}$, up to tensoring with $\det(V_{k})^{k}\otimes \det(V_{k}^{\prime})^{-k}$. 

\begin{lem}\label{lem:diff}
    For $i\geq 1$, there exists a unique extension of the differential 
    \[
    d_{i}:  p_{*} \cO_{\dI_{k-i}} \langle -i^{2}-i \rangle \longrightarrow p_{*} \cO_{\dI_{k-i+1}} \langle -(i-1)^{2} -(i-1) \rangle. 
    \]
\end{lem}

\begin{proof}
    Set $\tau=ba$. 
    Given the isomorphism (\ref{eq:xy}), it is equivalent to take the direct image of the sheaf $\cO_{\dJ_{k-i}} \otimes \det(V_{k-i})^{-k} \otimes \det(V_{k}^{\prime})^{k-i} [-k^{2}+ik]$.  
    The following argument is inspired by \cite{Cautis}.
    By definition, $\dJ_{k-i}$ is the locus $\{\tau\alpha_{k-i} = 0\}$ in $\dR_{k-i}$, so we have $\dim\ker(\tau)\geq k-i$.
    On the other hand, $\beta_{k-i}$ has image in $V_{k-i}$, so the inequality $\dim\ker(\beta_{k-i}) + \dim\ker(\tau) \geq k$ always holds over $\dJ_{k-i}$. 
    Consider the open subvariety  
    \begin{equation*}
        \iota: \dJ_{k-i}^{\circ} = \{ \dim\ker(\beta_{k-i}) + \dim\ker(\tau) \leq k+1 \} \hookrightarrow   \dJ_{k-i}. 
    \end{equation*}
    The complement of $\dJ_{k-i}^{\circ}$ in $\dJ_{k-i}$ is covered by the following loci: 
    (i) $\{ \on{rk}(\beta_{k-i}) \leq k-i-2\}$; 
    (ii) $\{ \dim\ker(\tau) \geq k-i+2 \}$; 
    (iii) $\{\on{rk}(\beta_{k-i}) =k-i-1,  \dim\ker(\tau) = k-i+1 \}$; cf.\ \cite[Lemma 3.1]{Cautis}. 
    Hence, the complement has codimension at least two by using the dimension formula of determinantal loci. 
    By the property of $S_2$  sheaves \cite[\S 2.5]{Cautis}, we have 
    \[
    p_{*}\cO_{\dJ_{k-i}} \cong (p\circ \iota)_{*} \cO_{\dJ_{k-i}^{\circ}}.
    \]
    Moreover, the partial resolution $p$ is isomorphic over $\dJ_{k-i}^{\circ}$ since we can recover $V_{k-i}$ either as the image of $\beta_{k}$ if it is of rank $k-i$, or as the kernel of $\tau$ otherwise. 
    With this, we can calculate  
    \begin{align}
        &\Hom\left(
        p_{*}\cO_{\dI_{k-i}}\langle -i^{2} -1\rangle, p_{*}\cO_{\dI_{k-i+1}} \langle -(i-1)^{2} \rangle \right) \notag \\
        \cong & \Hom\left(
        \cO_{\dJ_{k-i}^{\circ}}[-k], \cO_{\dJ_{k-i+1}^{\circ}} \otimes \det N_{\dJ_{k-i}^{\circ}\cap \dJ_{k-i+1}^{\circ}/\dJ_{k-i+1}^{\circ}}^{\vee}\right) \notag \\
        \cong & H^{0}\left(\cO_{\dJ_{k-i}^{\circ}\cap \dJ_{k-i+1}^{\circ}}\right),\label{eq:dcoh}
    \end{align}
    where in the first isomorphism, we used the formula
    \[
    \det N_{\dJ_{k-i}^{\circ}\cap \dJ_{k-i+1}^{\circ}/\dJ_{k-i+1}^{\circ}}^{\vee} \cong \det\left(V_{k}^{\prime}\otimes (V_{k-i+1}/V_{k-i})^{\vee}\right) \langle -2(k-i) \rangle,
    \]
    and the second isomorphism follows from \cite[Lemma 4.7]{CK}. 
    As the complement of $\Gr(\bC^{N},V_{k}^{\prime}) $ in $ [\Hom(\bC^{N},V_{k}^{\prime})/G_{k}]$ has codimension at least two, it is equivalent to calculate the cohomology group (\ref{eq:dcoh}) either relatively over $[\Hom(\bC^{N},V_{k}^{\prime})/G_{k}]$ or further restricted to $\Gr(\bC^{N},V_{k}^{\prime})$.
    \end{proof}

    \begin{rmk}\label{rmk:d!}
        In fact, the differential $d_{i}$ is uniquely determined up to a scalar.
        It is a $(\bC^{\times} \times \bC^{\times})$-equivariant map where
        the first $\bC^{\times}$ acts on $\gl_{k}$ by squared dilation (which also provides the grading on matrix factorizations),
        and the second $\bC^{\times}$-action is transported from the conical structure of the cotangent bundles of Grassmannians (i.e.\ induced by the scaling on $X_{k}$, see \cite[\S 2.4]{Cautis}) through the Kn{\"o}rrer periodicity. 
        Then, it is straightforward to see that
        \begin{equation}\label{eq:d!}
            d_{i} \in H^{0}_{\bC^{\times}\times \bC^{\times}} \left( \cO_{\dJ_{k-i}^{\circ}\cap \dJ_{k-i+1}^{\circ}} \right)    \cong   \bC. 
        \end{equation}
    \end{rmk}

Now, we have a complex of sheaves
\begin{equation*}
    p_{*} \cO_{\dI_{0}}\langle -k^{2}-k \rangle \xrightarrow{d_{k}} \cdots \xrightarrow{d_{i+1}} p_{*} \cO_{\dI_{k-i}} \langle -i^{2}-i \rangle \xrightarrow{d_{i}} \cdots \xrightarrow{d_{1}} \cO_{\dI_{k}}
\end{equation*}
on $\dR_{k}$, whose direct image under $\pi$ extends the twisted Rickard complex (\ref{eq:Ricker}).
Denote the iterative cones in its Postnikov system (\ref{eq:Pos})  by 
\[
\cT^{(i)} = \on{Conv}\left( p_{*} \cO_{\dI_{k-i}} \langle -i^{2}-i \rangle \xrightarrow{d_{i}} \cdots \xrightarrow{d_{1}} \cO_{\dI_{k}} \right), \quad i = 0, \cdots,k. 
\]
Taking into account the common tensor factor $\det(V_{k})^{k}\otimes \det(V_{k}^{\prime})^{-k}$, the main theorem is now equivalent to the statement that $(\on{id}\times \iota_{-})^{*}\pi_{*}\cT^{(k)}$ induces the equivalence $\bW_{[0,N)}$. 

\begin{lem}
    The convolution $\cT^{(k)}$ restricts to the diagonal $\Delta$ of $X^{+}_{k} \times \gl_{k}$, i.e.\ 
    \[
    (\on{id}\times \iota_{+})^{*} \pi_{*} \cT^{(k)} \cong \cO_{\Delta}. 
    \]
\end{lem}

\begin{proof}
    First note that under $\pi$, the image $(a\beta_{k},b)\in X_{k}^{+}$ iff $\beta_{k}$ is of full rank.
    As none of the images $p(\dI_{k-i})$ for $i>0$ contains any full rank $\beta_{k}$, it is reduced to consider $ (\on{id}\times \iota_{+})^{*} \pi_{*} \cO_{\dI_{k}}$. 
    When $\beta_{k}$ is invertible, the intertwining condition $\beta_{k}\epsilon_{k}^{\prime} = \epsilon_{k} \beta_{k}$ of $\dI_{k}$ implies $\epsilon_{k}^{\prime} = \beta_{k}^{-1}\epsilon_{k} \beta_{k}$. 
    So, for any such point in $\dI_{k-i}$, its image $(a \beta_{k}, b, \epsilon_{k}^{\prime})\in X_{k}\oplus \gl_{k}$ under $\pi$ is in the same $G_{k}$-orbit as $(a,\beta_{k} b, \epsilon_{k}) \in X_{k}^{+}\times \gl_{k}$. 
    This means, equivariantly, the base change of $\cO_{\dI_{k-i}}$ along $\on{id}\times \iota_{+}$ is isomorphic to $\cO_{\Delta}$. 
\end{proof}

It remains to verify that the essential image of the integral functor $\Phi_{\pi_{*} \cT^{(k)}}$ lies in the window category $\cW_{[0,N)}$.

\begin{proof}[Proof of Theorem~\ref{thm:main}]
    By Corollary~\ref{cor:resoln}, the sheaf $\cT^{(k)}$ has a locally free resolution $\cF^{\bullet}$ over $\dR_{k}$, whose terms are direct sums of Schur functors $\bS^{\mu} V_{k}^{\vee} \otimes \bS^{\la} V_{k}^{\prime}$ where each $\la = (\la_{i})_{i=1}^{k}$ satisfies $\la_{i} \in [0,k)$. 
    Let 
    \[X_{k}^{+}\times \gl_{k} \xlongleftarrow{p_{1}} (X_{k}^{+}\times \gl_{k}) \times (X_{k}\oplus \gl_{k}) \xlongrightarrow{p_{2}} X_{k}\oplus \gl_{k}
    \]
    be the projections. 
    For any matrix factorization $F = \{F_{0} \rightleftarrows F_{1}\}$ on $X_{k}^{+}\times \gl_{k}$, we need to show 
    \[
    p_{2,*} \left( p^{*}_{1} F \otimes \pi_{*} \cF^{\bullet} \right) \in \cW_{[0,N)}.
    \]
    The following argument partly follows \cite[\S 5.1]{BDF-2}. 
    Suppose $q: [X_{k}^{+}\times \gl_{k}/G_{k}] \to \Gr(k,\bC^{N})$ is the projection map. 
    Because $q$ is affine, objects of the form $q^{*}E$ for $E\in \Db(\Gr(k,\bC^{N}))$ generate $\Db(X_{k}^{+}\times \gl_{k})$, hence they also generate the components $F_{0}, F_{1}$. 
    So, it is enough to show
    \[
    (\Phi_{\pi_{*}\cF^{\bullet}}\circ q^{*}) (E) \in \cW_{[0,N)}.
    \]
    This composition functor has kernel $((q \times \on{id})\circ \pi)_{*} \cF^{\bullet}$, where 
    \begin{align*}
        (q\times \on{id})\circ \pi: [\dR_{k}/G_{k}\times G_{k}] &\longrightarrow \Gr(V_{k},\bC^{N}) \times [X_{k}\oplus \gl_{k}/G_{k}], \\
        (a,b,\beta_{k},\epsilon_{k},\epsilon_{k}^{\prime}) & \longmapsto (a,(a\beta_{k}, b,\epsilon_{k}^{\prime})). 
    \end{align*}
    By applying Theorem~\ref{thm:torsion} with $T=\Hom(V_{k}^{\prime},\bC^{N})$, $U=\Hom(V_{k}^{\prime},V_{k})$ and $G/P = \Gr(V_{k},\bC^{N})$ and using the projection formula, we can resolve each sheaf
    \[
    (p_{2}\circ (q \times \on{id}))_{*} \left( \pi_{*} (\bS^{\mu}V_{k}^{\vee} \otimes \bS^{\la}V_{k}^{\prime}) \otimes (q \circ p_{1})^{*} E \right)
    \]
    by a Lascoux resolution whose terms at degree $-n$ are  
    \begin{equation*}
        \bigoplus_{r \geq n} \bigoplus_{\nu} H^{r-n} \left( \Gr(V_{k},\bC^{N}), \bS^{\mu}V_{k}^{\vee}  \otimes
       \bS^{\nu^{\prime}} (\bC^{N}/V_{k})^{\vee}  \otimes E \right) \otimes \bS^{\la} V_{k}^{\prime} \otimes \bS^{\nu}V_{k}^{\prime}.
    \end{equation*}
    Here the second sum is over all Young diagrams $\nu= (\nu_{i})_{i=1}^{k}$ such that $|\nu| =r$ and $\nu_{i}\in [0,N-k]$.
    By the Littlewood--Richardson rule, the tensor product $\bS^{\la}V_{k}^{\prime} \otimes \bS^{\nu} V_{k}^{\prime}$ further decomposes into $\bigoplus_{\xi} c_{\la,\nu}^{\xi} \bS^{\xi} V_{k}^{\prime}$, where each $\xi$ satisfies $\xi_{i}\in [0,N)$, the half-open window. 
\end{proof}

\subsection{Grade restriction rules}\label{sec:2}
For $0\leq i\leq k$,  consider the locally closed subvariety 
\begin{align*}
     \ddR_{k-i} & =  \left\{
    (\tau, \alpha_{k-i},\beta_{k-i},\epsilon_{k-i},\epsilon_{k},\epsilon_{k}^{\prime}) \mid \alpha_{k-i} \text{ of full rank and } \epsilon_{k}|_{V_{k-i}} = \epsilon_{k-i} 
    \right\} \\
    \subset & \Hom(V_{k},V_{k}^{\prime}) \oplus \Hom(V_{k-i},V_{k})  
    \oplus  \Hom(V_{k}^{\prime},V_{k-i}) \oplus \gl(V_{k-i})\oplus \gl(V_{k}) \oplus \gl(V_{k}^{\prime}),
\end{align*}
cf.\ (\ref{eq:R}), and the map 
\begin{equation}\label{eq:qsm}
    \varphi: \dR_{k-i} \to \ddR_{k-i}, \quad 
    (a,b,\alpha_{k-i},\beta_{k-i},\epsilon_{k-i},\epsilon_{k},\epsilon_{k}^{\prime})  \mapsto (ba, \alpha_{k-i},\beta_{k-i},\epsilon_{k-i},\epsilon_{k},\epsilon_{k}^{\prime}). 
\end{equation}
Let $\ddI_{k-i}\subset \ddR_{k-i}$ be the intersection cut out by $\beta_{k-i}\epsilon_{k}^{\prime} =  \epsilon_{k-i}\beta_{k-i}$. 
Then, it is obvious that the previously defined $\cO_{\dI_{k-i}}$ is the pullback of $\cO_{\ddI_{k-i}}$ along $\varphi$. 
Moreover, the projection $p=p_{k-i}$ is constant over the directions $a:V_{k}\hookrightarrow \bC^{N}$ and $b: \bC^{N} \to V_{k}^{\prime}$, so it is also well-defined between the double-dotted spaces. 
As the morphism (\ref{eq:qsm}) is flat,  we have the base change
\[
\varphi^{*}p_{*}\cO_{\ddI_{k-i}} \cong p_{*} \varphi^{*} \cO_{\ddI_{k-i}} = p_{*} \cO_{\dI_{k-i}}.
\]
Throughout the definition of the potential $\upsilon_{k-i}$ and of the differential $d_{i}$, the maps $a$, $b$ always appear together as the composition $ba$. 
This means the maps $\upsilon_{k-i}$ and $d_{i}$ are also well-defined over the double-dotted spaces, and they are compatible with $\varphi$. 
More precisely, we have the LG model
\[
\upsilon_{k-i}: \ddR_{k-i} \to \bC, \quad 
(\tau,\alpha_{k-i},\beta_{k-i},\epsilon_{k-i},\epsilon_{k},\epsilon_{k}^{\prime})  \longmapsto \on{Tr} \left( \tau\alpha_{k-i}(\beta_{k-i}\epsilon_{k}^{\prime}- \epsilon_{k-i}\beta_{k-i}) \right),
\]
and the intersection $\ddJ_{k-i} \subset \ddR_{k-i}$ cut out by $\tau \alpha_{k-i}=0$. 
By the same argument as in  Lemma~\ref{lem:diff} and Remark~\ref{rmk:d!}, we can find a uniquely determined equivariant map
\[d_{i}:  p_{*} \cO_{\ddI_{k-i}} \langle -i^{2}-i \rangle \longrightarrow p_{*} \cO_{\ddI_{k-i+1}} \langle -(i-1)^{2} -(i-1) \rangle,\]
whose pullback along $\varphi$ coincides with the previously defined differential map. 
From now on, we will work over the double-dotted spaces, and by abuse of notation, we will retain the symbols $\upsilon_{k-i}$, $p=p_{k-i}$, and $d_{i}$ for these maps. 

\begin{rmk}
	The vector space $\ddR_{k}$ is a representation of either $G_{k}$, and we refer to the LG model $(\ddR_{k}, \upsilon_{k})$ as the \textit{quasi-symmetric model}.
\end{rmk}

Fix a basis $\{e_{1},\cdots,e_{k}\}$ of $V_{k}^{\prime}$. 
For each $1 \leq i \leq k$, consider the one-parameter subgroup
\begin{equation*}
    \gamma_{i}: \bC^{\times} \longrightarrow \GL(V_{k}^{\prime}), \quad t\longmapsto \on{diag}(1,\cdots,1,t,\cdots,t),
\end{equation*}
which has $k-i$ ones on the diagonal. 
These one-parameter subgroups parametrize a Kempf--Ness stratification $\{S^{(i)}\}$ of the unstable locus of $\ddR_{k}$ with respect to the stability condition $\theta_{+}$ on the $\GL(V_{k}^{\prime})$-action in the following way (see \cite[Lemma 6.1.9]{Toda}).
\begin{enumerate}[wide, labelindent=0pt, label = (\alph*)]
    \item The stratum $S^{(i)}$ consists of points $(\tau,\beta_{k},\epsilon_{k},\epsilon_{k}^{\prime}) \in \ddR_{k}$ where the image of $\beta_{k}^{\vee}$ generates a $(k-i)$-dimensional $\bC[(\epsilon_{k}^{\prime})^{\vee}]$-submodule $\langle \beta_{k}^{\vee}(V_{k}^{\vee}) \rangle$ of $(V_{k}^{\prime})^{\vee}$. 
    Note that the closure of $S^{(i)}$ contains all lower rank strata $\on{cl}(S^{(i)}) = \bigcup_{j\geq i} S^{(j)}$. 
    
    \item Each $S^{(i)}$ is the $G_{k}$-orbit of the attracting locus $Y^{(i)}$, which consists of points $(\tau,\beta_{k},\epsilon_{k},\epsilon_{k}^{\prime})$  where 
    $\langle  \beta_{k}^{\vee}(V_{k}^{\vee}) \rangle = \on{span}_{\bC}\{e_{1}^{\vee},\cdots,e_{k-i}^{\vee}\}$.
    The attracting locus $Y^{(i)}$ itself is equivariant with respect to the parabolic subgroup $P_{i}$ of $G_{k}$ that preserves $\on{span}_{\bC}\{e_{1},\cdots,e_{k-i}\}^{\perp}$.

    \item By definition, the attracting locus $Y^{(i)}$ flows into the fixed locus $Z^{(i)}$ under $\gamma_{i}(t)$ when $t\to 0$.
    A point $(\tau,\beta_{k},\epsilon_{k},\epsilon_{k}^{\prime})$ is in $Z^{(i)}$ iff
    $\langle \beta_{k}^{\vee}(V_{k}^{\vee}) \rangle = \on{span}_{\bC}\{e_{1}^{\vee},\cdots,e_{k-i}^{\vee}\}$, $\tau(V_{k}) \subset \on{span}_{\bC}\{e_{1},\cdots,e_{k-i}\}$, and $\epsilon_{k}^{\prime}\in \on{Lie}(L_{i})$ for the Levi subgroup $L_{i}$ that preserves $\on{span}_{\bC}\{e_{1},\cdots,e_{k-i}\} \oplus \on{span}_{\bC}\{e_{1},\cdots,e_{k-i}\}^{\perp}$.
\end{enumerate}
In summary, we have a diagram 
\begin{equation*}
    \begin{tikzcd} 
        S^{(i)} \arrow[d,shift left,"\pi_{i}"] \arrow[r,hook,"q_{i}"] & (\bigcup_{j>i} S^{(j)})^{\complement}  & (\bigcup_{j\geq i} S^{(j)})^{\complement} \arrow[l, hook',"\iota_{i}"'] \\
        Z^{(i)} \arrow[ur, hook, "\sigma_{i}"] \arrow[u,hook,shift left,"\sigma_{i}"] & &
    \end{tikzcd}
\end{equation*}
where $\pi_{i}$ takes the attracting limit, $q_{i}$, $\sigma_{i}$ are embeddings\footnote{We follow \cite{HL,HLShip} to denote the composition $q_{i} \circ \sigma_{i}$ also by $\sigma_{i}$, which will not cause any confusion in practice.}, and $\iota_{i}$ is an inclusion of open subset.
The diagram is also equivariant with respect to the $\GL(V_{k})$-action, so we will work with this additional $G_{k}$-equivariance in the following. 

Since $Z^{(i)}$ is fixed by $\gamma_{i}$, there is a decomposition  
\begin{equation*}
    \MF_{ G_{k} \times L_{i}\times  \bC^{\times}}(Z^{(i)}, \upsilon_{k}) = \bigoplus_{\kappa\in\mathbb{Z}} \MF_{ G_{k} \times L_{i}\times \bC^{\times}}(Z^{(i)}, \upsilon_{k})_{\kappa}
\end{equation*}
into full subcategories of matrix factorizations of various $\gamma_{i}$-weights $\kappa$. 
Following \cite{HL}, we denote by $(\cdot)_{\kappa}$ the exact functor that takes the $\gamma_{i}$-weight $\kappa$ summand from this decomposition.  
Let $\cA_{\kappa}$ be the essential image of the fully faithful functor $q_{i,*} \circ \pi_{i}^{*}$ (see, e.g., \cite[Theorem 6.1.2]{Toda})
\begin{equation*}
    \MF_{G_{k} \times  L_{i}\times \bC^{\times}}(Z^{(i)}, \upsilon_{k} )_{\kappa} \xlongrightarrow{\pi_{i}^{*}} \MF_{G_{k}^{2}\times \bC^{\times}}(S^{(i)},\upsilon_{k}) \xlongrightarrow{q_{i,*}} \MF_{G_{k}^{2} \times \bC^{\times}}((\bigcup_{j>i} S^{(j)})^{\complement}, \upsilon_{k}). 
\end{equation*}
Let $\eta_{i}$ be the sum of the $\gamma_{i}$-weights of the conormal bundle $N_{S^{(i)}/\ddR_{k}}^{\vee}$, which is given by the number
\[
\eta_{i} = \on{wt}_{\gamma_{i}} \det N^{\vee}_{Y^{(i)}/\ddR_{k}} - \on{wt}_{\gamma_{i}} \det \gl(V_{k}^{\prime})^{\gamma_{i}>0} = ik
\]
from \cite[Equation (4)]{HL}. 
Now, the inverse of the equivalence 
\begin{equation}\label{eq:ZA}
    q_{i,*} \pi_{i}^{*}:  \MF_{G_{k} \times  L_{i}\times \bC^{\times}}(Z^{(i)}, \upsilon_{k} )_{\kappa_{i}} \simrightarrow \cA_{\kappa_{i}}
\end{equation}
is given by $(\sigma_{i}^{*}(\cdot))_{\kappa_{i}}$ or $(\sigma_{i}^{*}(\cdot))_{\kappa_{i}+\eta_{i}} \otimes \omega_{q_{i}}$; see \cite[Lemma 2.2]{HLShip}. 

The following categorical Kirwan surjectivity theorem is due to \cite{HL,BFK}. 
\begin{thm}[{\cite[Theorem 6.1.2]{Toda}}]\label{thm:SOD}
    For $\kappa_{i}\in \mathbb{Z}$, there is a semi-orthogonal decomposition 
    \begin{equation*}
        \MF_{G_{k}^{2} \times \bC^{\times}} ((\bigcup_{j>i} S^{(j)})^{\complement}, \upsilon_{k}) = \langle \cA_{<\kappa_{i}}, \cG_{\kappa_{i}}, \cA_{\geq \kappa_{i}} \rangle,
    \end{equation*} 
    where $\cG_{\kappa_{i}}$ is the full subcategory of matrix factorizations $F$ such that 
    \begin{equation*}
        \sigma_{i}^{*}F \in \MF_{G_{k} \times L_{i}\times  \bC^{\times}}(Z^{(i)}, \upsilon_{k} )_{[\kappa_{i},\kappa_{i}+\eta_{i})}. 
    \end{equation*}
    Moreover, the restriction $\iota_{i}^{*}: \cG_{\kappa_{i}} \to \MF_{G_{k}^{2} \times \bC^{\times}}((\bigcup_{j\geq i} S^{(j)})^{\complement}, \upsilon_{k})$ is an equivalence. 
\end{thm}
As a result, for a choice of integers $\kappa_{\bullet} = (\kappa_{1},\cdots,\kappa_{k})\in \mathbb{Z}^{k}$, we obtain a nested semi-orthogonal decomposition 
\[
\MF_{G_{k}^{2} \times \bC^{\times}}(\ddR_{k},\upsilon_{k}) = \langle\cA_{< \kappa_{k}},\cdots, \cA_{< \kappa_{1}}, \cG_{\kappa_{\bullet}}, \cA_{\geq \kappa_{1}}, \cdots, \cA_{\geq \kappa_{k}} \rangle, 
\]
where $\cG_{\kappa_{\bullet}}$ consists of matrix factorizations $F$ that satisfy the \textit{grade restriction rule} 
\[
\on{wt}_{\gamma_{i}} \sigma_{i}^{*}F \in [\kappa_{i},\kappa_{i}+\eta_{i}), \quad i=1,\cdots,k. 
\]

\subsection{Rickard complex as mutations}\label{sec:3}

Choose $\kappa_{\bullet} = (0,\cdots,0)\in \mathbb{Z}^{k}$ from now on.
Consider the full subcategory 
\[
\cC_{\kappa_{i}} \subset \MF_{G_{k}^{2} \times \bC^{\times}}((\bigcup_{j>i} S^{(j)})^{\complement}, \upsilon_{k})
\]
of matrix factorizations $F$ such that 
\begin{equation*}
    \sigma_{i}^{*}F \in \MF_{G_{k} \times L_{i}\times  \bC^{\times}}(Z^{(i)}, \upsilon_{k} )_{[\kappa_{i},\kappa_{i}+\eta_{i}]}. 
\end{equation*}
By Theorem~\ref{thm:SOD}, it has a semi-orthogonal decomposition
\begin{equation*}
    \cC_{\kappa_{i}} = \langle \cG_{\kappa_{i}}, \cA_{\kappa_{i}} \rangle. 
\end{equation*}
This means, for any matrix factorization $F\in \cC_{\kappa_{i}}$, there exits a unique morphism $E \to F$, with $E \in \cA_{\kappa_{i}}$, whose cone is contained in $\cG_{\kappa_{i}}$. 
From \cite{HL} (or \cite[Lemma 2.3]{HLShip}), we know this is realized by the adjunction
\begin{equation}\label{eq:mut}
     q_{i,*} \pi_{i}^{*} \left((\sigma_{i}^{*} F)_{\kappa_{i}+\eta_{i}} \otimes \omega_{q_{i}} \right) \xlongrightarrow{\varepsilon} F  \longrightarrow \on{Cone}(\varepsilon),
\end{equation}
where the left end lies in $ \cA_{\kappa_{i}}$ and the right end lies in $\cG_{\kappa_{i}}$. 

\begin{lem}\label{lem:resolni}
    For $i\geq 0$, the sheaf $p_{*}\cO_{\ddI_{k-i}}$ has a locally free resolution $\cF_{i}^{\bullet}$ over $\ddR_{k}$, whose terms are direct sums of Schur functors $\bS^{\mu}V_{k}^{\vee}\otimes \bS^{\la}V_{k}^{\prime}$ where the highest weight $\la$ satisfies $k\geq \la_{1} \geq \cdots \geq \la_{k} \geq 0$. 
\end{lem}

\begin{proof}
    The local complete intersection $\ddI_{k-i}$ is cut out in $\ddR_{k-i}$ by the intertwining condition $\beta_{k-i}\epsilon_{k}^{\prime} = \epsilon_{k-i}\beta_{k-i}$.
    So, we can resolve $\cO_{\ddI_{k-i}}$ by a Koszul complex $\cK^{\bullet} = \extp^{\bullet} (V_{k}^{\prime}\otimes V_{k-i}^{\vee})$ over $\ddR_{k-i}$. 
    The quotient $[\ddR_{k-i}/ G_{k-i} \times G_{k}^{2}]$ is the subspace of 
    \[
    [\Hom(V_{k},V_{k}^{\prime}) \times \Hom(V_{k}^{\prime},V_{k-i}) \times \gl(V_{k}) \times \gl(V_{k}^{\prime}) /G_{k}\times G_{k}]
    \]
    over $\Gr(V_{k-i},V_{k})$
    where $\epsilon_{k}\in \gl(V_{k})$ preserves $V_{k-i}\hookrightarrow V_{k}$. 
    By applying Theorem~\ref{thm:torsion} with
    \[
    T= \Hom(V_{k}^{\prime},V_{k})\oplus \gl(V_{k}), \quad U = \Hom(V_{k}^{\prime},V_{k-i}) \oplus \on{Lie}(P), \quad G/P = \Gr(V_{k-i},V_{k}), 
    \]
    we obtain a Lascoux resolution of each shaef $p_{*}(\cK^{\bullet})$ with terms at degree $-n$ given by
    \[
    \bigoplus_{r\geq n} H^{r-n} \left( \Gr(V_{k-i},V_{k}),  \cK^{\bullet} \otimes \extp^{r} \left( 
    V_{k}^{\prime} \otimes (V_{k}/V_{k-i})^{\vee} \oplus
    V_{k-i}\otimes (V_{k}/V_{k-i})^{\vee}   \right) \right).
    \]
    After decomposing the exterior power and using the projection formula, the total complex has general terms of the form
    \[
    H^{\bullet}_{\Gr(k-i,k)} \left(\bS^{\alpha^{\prime}} V_{k-i}^{\vee} \otimes \bS^{\beta}V_{k-i} \otimes \bS^{\beta^{\prime}}  (V_{k}/V_{k-i})^{\vee} \otimes \bS^{\gamma^{\prime}}  (V_{k}/V_{k-i})^{\vee}   \right) \otimes \bS^{\alpha}V_{k}^{\prime} \otimes \bS^{\gamma} V_{k}^{\prime}, 
    \]
    where the Young diagrams $\alpha$, $\beta$ and $\gamma$ have at most $k-i$, $i$ and $i$ boxes in each of their rows respectively. 
    When it is nonvanishing, the cohomology group is isomorphic to a direct sum of Schur functors of $V_{k}^{\vee}$ by the Borel--Weil--Bott theorem \cite[\S 4.1]{Wey}. 
    According to the Littlewood--Richardson rule, the tensor product $ \bS^{\alpha}V_{k}^{\prime} \otimes \bS^{\gamma} V_{k}^{\prime}$ has a decomposition $\bigoplus_{\la} c_{\alpha,\gamma}^{\la} \bS^{\la} V_{k}^{\prime}$, where each $\la$ has at most $k$ boxes in each of its rows. 
\end{proof}

As we said in \S\ref{sec:kernel}, each sheaf $p_{*}\cO_{\ddI_{k-i}}$ is supported on the union of components
\[
\ddot{C}_{\leq k-i}  =  \{ \on{rk}(\beta_{k}) \leq k-i \} \subset \ddI_{k}. 
\]
Since $\ddI_{k}$ is defined by the equation $\beta_{k}\epsilon_{k}^{\prime} = \epsilon_{k} \beta_{k}$, the kernel of $\beta_{k}$ is preserved by $\epsilon_{k}^{\prime}$, and hence the image of $\beta_{k}^{\vee}$ is preserved by $(\epsilon_{k}^{\prime})^{\vee}$. 
This implies that the components $\ddot{C}_{\leq k-i}$ are contained in $\on{cl}(S^{(i)}) = \bigcup_{j\geq i} S^{(j)}$. 

\begin{lem}\label{lem:win} 
    For $i\geq l\geq 0$, the restriction of $p_{*}\cO_{\ddI_{k-l}}$ to $(\bigcup_{j>i}S^{(i)})^{\complement}$ lies in $\cC_{\kappa_{i}}$. 
    In particular, when $i=l>0$, the restriction further belongs to $\cA_{\kappa_{i}}$. 
\end{lem}

\begin{proof}
    The terms of the resolution $\cF_{l}^{\bullet}$ of $p_{*}\cO_{\ddI_{k-l}}$ from Lemma~\ref{lem:resolni} are direct sums of $\bS^{\mu}V_{k}^{\vee}\otimes \bS^{\la}V_{k}^{\prime}$, where $k\geq \la_{1}\geq \cdots \geq \la_{k}\geq 0$. 
    It is straightforward to see that each pullback $\sigma_{i}^{*}(\bS^{\mu}V_{k}^{\vee} \otimes \bS^{\la}V_{k}^{\prime})$ has its $\gamma_{i}$-weights located in the range $[0,ik]$. 
    In particular, as $p_{*}\cO_{\ddI_{k-i}}$ is supported on $\on{cl}(S^{(i)})$, its restriction to $(\bigcup_{j>i}S^{(j)})^{\complement}$ is supported on $S^{(i)}$, hence lies in $\cA_{\kappa_{i}}$. 
\end{proof}

\begin{prop}\label{prop:main}
    For $i>0$, the restriction of $\cT^{(i)}$ to $(\bigcup_{j>i}S^{(j)})^{\complement}$ lies in $\cG_{\kappa_{i}}$. 
\end{prop}

\begin{proof}
    Everything in this proof is pulled back to the open subset $U=(\bigcup_{j>i}S^{(j)})^{\complement}$.
    Occasionally, we will omit the restriction or pullback symbol $(\cdot)|_{U}$.

    For any point $(\tau,\beta_{k},\epsilon_{k},\epsilon_{k}^{\prime})$ of $S^{(i)}$, the image of $\beta_{k}^{\vee}$ generates a $(k-i)$-dimensional $\bC[(\epsilon_{k}^{\prime})^{\vee}]$-submodule of $(V_{k}^{\prime})^{\vee}$. 
    On the other hand, the intertwining condition $\beta_{k}\epsilon_{k}^{\prime} = \epsilon_{k} \beta_{k}$ implies that the image of $\beta_{k}^{\vee}$ is preserved by $(\epsilon_{k}^{\prime})^{\vee}$. 
    Thus, the partial resolution $p=p_{k-i}:\ddI_{k-i}|_{U} \to U$ factors through the stratum $q_{i}: S^{(i)} \hookrightarrow {U}$ via the base change map $p|_{S^{(i)}}:\ddI_{k-i}|_{U}\to S^{(i)}$.
    It is actually isomorphic over its image since we can recover $V_{k-i}$ as the image of $\beta_{k}$. 
    Consider the scheme-theoretic fibre product of $S^{(i)}$ and $\ddI_{k-l}|_{U}$ over $U$
    \begin{equation}\label{eq:Q}
        \begin{tikzcd}
            Q \arrow[r, "\tilde{q}"]  \arrow[d,"\tilde{p}"]  &  
            \ddI_{k-l}|_{U}  \arrow[dd,"p_{k-l}"] \\
            \ddI_{k-i}|_{U} \arrow[d, "p|_{S^{(i)}}"] & 
             \\
            S^{(i)}  \arrow[r,hook,"q_{i}"] &
            U,
        \end{tikzcd}
    \end{equation}
    where the left vertical arrow factors through $\ddI_{k-i}|_{U}$ because of the coincidence of the supports of $p_{k-l}(\ddI_{k-l}) \cap S^{(i)}$ and  $p(\ddI_{k-i}|_{U})$.
    The fibre product $Q$ is a correspondence between $\ddI_{k-i}|_{U}$ and $\ddI_{k-l}|_{U}$ with natural projections $\tilde{p}$ and $\tilde{q}$.
    It can be identified with the locally closed subvariety
        \begin{align*}
            & Q = \left\{
            (\alpha_{k-i},\alpha_{k-l},\beta_{k-i},\tau,\epsilon_{k-i},\epsilon_{k-l},\epsilon_{k},\epsilon_{k}^{\prime}) \mathrel{\bigg|}
            \begin{array}{c}
            \alpha_{k-i}, \alpha_{k-l}, \beta_{k-i} \text{ of full rank,} \\ 
            \epsilon_{k}|_{V_{k-l}} = \epsilon_{k-l},\, \epsilon_{k-l}|_{V_{k-i}} = \epsilon_{k-i},\\
            \beta_{k-i}\epsilon_{k}^{\prime}  = \epsilon_{k-i} \beta_{k-i}
        	\end{array}
            \right\}
        \end{align*}
    inside the affine space
    \begin{align*}
        \Hom(V_{k-i},V_{k-l})  & \oplus   \Hom(V_{k-l},V_{k}) \oplus \Hom(V_{k}^{\prime},V_{k-i}) \\    & \oplus  \Hom(V_{k},V_{k}^{\prime}) \oplus \gl(V_{k-i}) \oplus \gl(V_{k-l}) \oplus \gl(V_{k}) \oplus  \gl(V_{k}^{\prime}).
    \end{align*}
    Now, recall that we are interested in the summand 
    \begin{align*}
        (\sigma_{i}^{*} p_{*}\cO_{\ddI_{k-l}})_{\kappa_{i}+\eta_{i}} \otimes \omega_{q_{i}} & \cong (\sigma_{i}^{*} q_{i}^{!} p_{k-l,*}\cO_{\ddI_{k-l}})_{\kappa_{i}}.
    \end{align*}
    If the fibre product (\ref{eq:Q}) were taken in the derived sense (e.g., as a dg scheme), then the derived base change formula 
    \[
    q_{i}^{!} p_{k-l,*} \simeq (p|_{S^{(i)}} \circ \tilde{p})_{*} \tilde{q}^{!}
    \]
    is directly applicable. 
    However, for the purpose of computing the summand $(\cdot)_{\kappa_{i}}$, the classical truncation $Q$ is sufficient:
    as the higher Tor sheaves $\underline{\on{Tor}}_{\bullet}^{U}(\cO_{S^{(i)}}, \cO_{\ddI_{k-l}})$ can be computed by $\cO_{\ddI_{k-l}} \otimes \extp^{\bullet} N_{S^{(i)}/U}^{\vee}$ where the conormal sheaf $N_{S^{(i)}/U}^{\vee}$ has positive $\gamma_{i}$-weights, there is an isomorphism
    \begin{equation} \label{eq:sumd}
         (\sigma_{i}^{*} q_{i}^{!} p_{k-l,*}\cO_{\ddI_{k-l}})_{\kappa_{i}} \cong \left(\sigma_{i}^{*} (p|_{S^{(i)}} \circ \tilde{p})_{*} \cO_{Q}\langle -2i(i-l) \rangle \right)_{\kappa_{i}}. 
    \end{equation}
    Here we have used:
    (i) the formula $\omega_{\tilde{q}} \cong \cO_{Q}\langle -2i(i-l) \rangle$, which can be computed in the same way as in the proof of Proposition~\ref{prop:Dc};
    (ii) the fact that $(p|_{S^{(i)}} \circ \tilde{p})_{*} \cO_{Q}$ is a direct sum of shifts of $(p|_{S^{(i)}})_{*}\cO_{\ddI_{k-i}}$ (see (\ref{eq:pQ}) below), which only has $\gamma_{i}$-weight $\kappa_{i}$ when restricted from $S^{(i)}$ to $Z^{(i)}$ by Lemma~\ref{lem:win}.
    
    Next we calculate the direct image of $\cO_{Q}$ along the projection $\tilde{p}$, which is given by forgetting $V_{k-l}$ and composing $\alpha_{k-l}\circ \alpha_{k-i}$.
    Let $T$ be the parabolic subalgebra of $\epsilon_{k}\in \gl(V_{k})$ that preserves $\alpha_{k-i}:V_{k-i}\hookrightarrow V_{k}$, and $U\subset T$ the subspace that further preserves the flag $V_{k-i} \hookrightarrow V_{k-l} \hookrightarrow V_{k}$. 
    By applying Theorem~\ref{thm:torsion}, we can resolve the torsion sheaf $\tilde{p}_{*}\cO_{Q}$ by the Lascoux resolution
    \[
    \cF^{-n} =\bigoplus_{r\geq n} H^{r-n} \left( \Gr(V_{k-l}/V_{k-i}, V_{k}/V_{k-i}), \extp^{r} (V_{k-l}/V_{k-i}) \otimes (V_{k}/V_{k-l})^{\vee} \langle -2r\rangle \right). 
    \]
    This is exactly the de Rham cohomology of $\Gr(i-l,i)$, so  
    \begin{equation}\label{eq:pQ}
        \tilde{p}_{*}\cO_{Q} \cong \bigoplus_{r =0}^{l(i-l)} (\cO_{\ddI_{k-i}})^{\oplus b_{l,r}} \langle -2r \rangle, 
    \end{equation}
    where $b_{l,r}$ is the Betti number $\dim H^{2r}(\Gr(i-l,i))$ (note that it is also indexed by Young diagrams of $r$ boxes with at most $i-l$ rows and $l$ columns). 
    Thus, the summand (\ref{eq:sumd}) is isomorphic to a direct sum of shifts of $(\sigma_{i}^{*}(p|_{S^{(i)}})_{*}\cO_{\ddI_{k-i}})_{\kappa_{i}}$, with multiplicities and degrees parametrized by the cohomology ring of $\Gr(i-l,i)$. 
    In Lemma~\ref{lem:win}, we have shown $p_{*}\cO_{\ddI_{k-i}}\in \cA_{\kappa_{i}}$. 
    So, the later summand just mentioned is the inverse of $p_{*}\cO_{\ddI_{k-i}}$ under the equivalence (\ref{eq:ZA}), namely
    \begin{equation*}
        p_{*}\cO_{\ddI_{k-i}} \cong q_{i,*}\pi_{i}^{*} (\sigma_{i}^{*}p_{*}\cO_{\ddI_{k-i}})_{\kappa_{i}} 
        \cong q_{i,*}\pi_{i}^{*} (\sigma_{i}^{*}(p|_{S^{(i)}})_{*}\cO_{\ddI_{k-i}})_{\kappa_{i}}. 
    \end{equation*}
    We are now ready to spell out the counit (\ref{eq:mut}) on the convolution
    \[
    \cT^{(i-1)} = \on{Conv}\left( p_{*}\cO_{\ddI_{k-i+1}}\langle -(i-1)^{2} -(i-1) \rangle \xlongrightarrow{d_{i-1}} \cdots \xlongrightarrow{d_{l+1}}  p_{*}\cO_{\ddI_{k-l}}\langle -l^{2} -l \rangle \xlongrightarrow{d_{l}} \cdots  \right).
    \]
    For each $l\geq i$, we have 
    \begin{align}
        & q_{i,*}\pi_{i}^{*}( (\sigma_{i}^{*}p_{*}\cO_{\ddI_{k-l}}\langle -l^{2}-l \rangle )_{\kappa_{i}+\eta_{i}} \otimes \omega_{q_{i}} ) \notag \\
        \cong &  
        \bigoplus_{r =0}^{l(i-l)} (p_{*}\cO_{\ddI_{k-i}})^{\oplus  b_{l,r}}  \langle -l^{2}-l- 2i(i-l)- 2r  \rangle. \label{eq:copy}
    \end{align}
    On the quasi-symmetric model $\ddR_{k}$ (which is an affine space), the $\bC^{\times}$-equivariant differentials $d_{l}$ are given by maps between these copies of $p_{*}\cO_{\ddI_{k-i}}$ with matching degree shifts. 
    By the uniqueness (\ref{eq:d!}), they must be identities up to scalar multiples. 
    More specifically, each copy of $p_{*}\cO_{\ddI_{k-i}}$ in $(p_{*}\cO_{\ddI_{k-i}})^{\oplus  b_{l,r}}$ corresponds to a Schubert cycle in $H^{2r}(\Gr(i-l,i))$, indexed by a Young diagram $\la = (\la_{1},\cdots,\la_{i-l})$ with $l \geq \la_{1} \geq \cdots \geq \la_{i-l} \geq 0$.
    When the first column of $\la$ has $i-l$ boxes (i.e.\ $\la_{i-l}>0$), there is a unique Schubert cycle in $H^{2(r-i+l)}(\Gr(i-l+1,i))$, which corresponds to the Young diagram $(\la_{1}-1,\cdots,\la_{i-l}-1,0)$ obtained from  $\la$ by deleting its first column.    
    Conversely, when $\la_{i-l}=0$, there is a unique cycle in $H^{2(r+i-l-1)}(\Gr(i-l-1,i))$, which corresponds to the Young diagram $(\la_{1}+1,\cdots,\la_{i-l-1}+1)$. 
    Deleting the first column of this later Young diagram gives back $\la$. 
    By a simple calculation, one can see that the copies of $p_{*}\cO_{\ddI_{k-i}}$  corresponding to these Young diagrams have matching degrees in (\ref{eq:copy}), i.e.\ 
    \begin{align*}
        -l^{2} - l -2i(i-l) -2r & = -(l-1)^{2} - (l-1) -2i(i-l+1) -2(r-i+l)\\
        & =  -(l+1)^{2} - (l+1) -2i(i-l-1) -2(r+i-l-1). 
    \end{align*}
    After cancelling these summands in the convolution, we are left with only one copy 
    \[
    q_{i,*}\pi_{i}^{*}( (\sigma_{i}^{*} \cT^{(i-1)} )_{\kappa_{i}+\eta_{i}} \otimes \omega_{q_{i}}) \cong p_{*}\cO_{\ddI_{k-i}}\langle -i^{2} -1 \rangle,
    \]
    which corresponds exactly to $H^{0}(\Gr(1,i))$,
    and the mutation (\ref{eq:mut}) reads as 
    \[
    p_{*}\cO_{\ddI_{k-i}}\langle -i^{2} -1 \rangle \xlongrightarrow{\varepsilon} \cT^{(i-1)} \longrightarrow \on{Cone}(\varepsilon). 
    \]
    Compare it to the exact triangle
    \[
    p_{*}\cO_{\ddI_{k-i}}\langle -i^{2} -1 \rangle  \xlongrightarrow{d_{i}} \cT^{(i-1)} \longrightarrow \cT^{(i)}
    \]
    in the Postnikov system (\ref{eq:Pos}), we claim $\cT^{(i)} \in\cG_{\kappa_{i}}$ by the uniqueness (\ref{eq:d!}) of $d_{i}$. 
\end{proof}

\begin{cor}
    The convolution $\cT^{(k)}$ 
    satisfies the grade restriction rule. 
\end{cor}

\begin{proof}
    For each $i>0$, the sheaves $p_{*}\cO_{\ddI_{k-j}}$ for $j>i$ are supported on $\bigcup_{j>i}S^{(j)}$. 
    Hence, the restriction of $\cT^{(k)}$ to $(\bigcup_{j>i}S^{(j)})^{\complement}$ is  same as that of $\cT^{(i)}$, which lies in $\cG_{\kappa_{i}}$ by Proposition~\ref{prop:main}. 
\end{proof}

The quasi-symmetric model $\ddR_{k}$ is a representation of $\GL(V_{k}^{\prime})$. 
By applying the magic window theory (\S\ref{sec:magic}) to this $G_{k}$-action (and equivariantly with respect to the other $G_{k}$), we obtain a magic window subcategory 
\[
\mathcal{M}_{[0,k)} \subset \MF_{G_{k}^{2}\times \bC^{\times}}(\ddR_{k},\upsilon_{k})
\]
of equivariant matrix factorizations whose components are generated by $\bS^{\la}V_{k}^{\prime}$ with $\la_{i}\in [0,k)$. 
As a result of the general theory, this magic window coincides with our grade restriction window
\[\mathcal{M}_{[0,k)} = \cG_{\kappa_{\bullet}}, \]
see, e.g., \cite[Proposition 6.1.6]{Toda} and references therein. 

\begin{cor}\label{cor:resoln}
    The convolution $\cT^{(k)}$ has a locally free resolution $\cF^{\bullet}$ over $\ddR_{k}$, whose terms are direct sums of Schur functors $\bS^{\mu}V_{k}^{\vee} \otimes \bS^{\la}V_{k}^{\prime}$ where the highest weight $\la$ satisfies $k > \la_{1} \geq \cdots \geq \la_{k} \geq 0$. 
\end{cor}

\begin{proof}
    Take the convolution of the resolutions $\cF_{i}^{\bullet}$ for $i=0,\cdots,k$ from Lemma~\ref{lem:resolni}, and then resolve those $\bS^{\la}V_{k}^{\prime}$ with $\la_{1}=k$ by using the magic window generators of $\mathcal{M}_{[0,k)}$. 
    See \cite[\S 3.3]{HLS} for details on an algorithm to do this in practice. 
\end{proof}

The pullback of this resolution $\cF^{\bullet}$ along the flat morphism $\varphi:\dR_{k} \to \ddR_{k}$ (\S\ref{sec:2}) will be the one used in our proof of the main theorem in \S\ref{sec:1}.

\printbibliography

\vspace{1em}
\begin{flushleft}
    {\fontsize{9.5}{11}\selectfont Department of Mathematics, Imperial College, London, SW7 2AZ, United Kingdom \\
    \href{mailto:w.zhou21@imperial.ac.uk}{\texttt{w.zhou21@imperial.ac.uk}}}
\end{flushleft}

\end{document}